\documentclass[11pt,reqno,tbtags]{amsart}
\usepackage{amssymb}
\usepackage{natbib}
\bibpunct[, ]{[}{]}{;}{n}{,}{,}
\usepackage{url}
\numberwithin{equation}{section}
\allowdisplaybreaks

\newtheorem{theorem}{Theorem}[section]
\newtheorem{lemma}[theorem]{Lemma}

\newtheorem{corollary}[theorem]{Corollary}

\theoremstyle{definition}
\newtheorem{example}[theorem]{Example}

\newtheorem{problem}[theorem]{Problem}
\newtheorem{problems}[theorem]{Problems}
\newtheorem{remark}[theorem]{Remark}
\newtheorem*{remark*}{Remark}

\newtheorem*{acks}{Acknowledgements}

\theoremstyle{remark}

\newenvironment{romenumerate}{\begin{enumerate}% gives (i), (ii) etc.
 }{\end{enumerate}}

\newcounter{oldenumi}
\newenvironment{romenumerateq}% continues numbering from previous romenumerate
{\setcounter{oldenumi}{\value{enumi}}
\begin{romenumerate} \setcounter{enumi}{\value{oldenumi}}}
{\end{romenumerate}}

\newcounter{thmenumerate}
\newenvironment{thmenumerate}
{\setcounter{thmenumerate}{0}%
 \def\item{\par% \ifnum\thethmenumerate=0\else\par\fi %\noindent\fi
 \refstepcounter{thmenumerate}\textup{(\roman{thmenumerate})\enspace}}
}
{}

\newcounter{xenumerate}   %no left indentation; thus wider lines

\newcommand{\refT}[1]{Theorem~\ref{#1}}
\newcommand{\refC}[1]{Corollary~\ref{#1}}
\newcommand{\refL}[1]{Lemma~\ref{#1}}
\newcommand{\refR}[1]{Remark~\ref{#1}}
\newcommand{\refS}[1]{Section~\ref{#1}}
\newcommand{\refSS}[1]{Section~\ref{#1}}

\newcommand{\refE}[1]{Example~\ref{#1}}
\newcommand{\refF}[1]{Figure~\ref{#1}}

\begingroup
  \divide\count255 by 60
  \multiply\count255 by -60
  \advance\count255 by \time
  \ifnum \count255 < 10 \xdef\klockan{\the\count1.0\the\count255}
  \else\xdef\klockan{\the\count1.\the\count255}\fi
\endgroup

\newcommand\nopf{\qed}   % for theorem without proof
\newcommand\set[1]{\ensuremath{\{#1\}}}
\newcommand\bigset[1]{\ensuremath{\bigl\{#1\bigr\}}}

\newcommand\bigpar[1]{\bigl(#1\bigr)}

\newcommand\Bigcpar[1]{\Bigl\{#1\Bigr\}}

\def\rompar(#1){\textup(#1\textup)}    % usage: \rompar(...)

\def\xexp(#1){e^{#1}}

\newcommand\floor[1]{\lfloor#1\rfloor}

\newcommand\ntoo{\ensuremath{{n\to\infty}}}

\newcommand\iid{i.i.d.\spacefactor=1000}    
\newcommand\ie{i.e.\spacefactor=1000}
\newcommand\eg{e.g.\spacefactor=1000}

\newcommand\cf{cf.\spacefactor=1000}
\newcommand{\as}{a.s.\spacefactor=1000}
\newcommand{\aex}{a.e.\spacefactor=1000}
\newcommand\ii{\mathrm{i}}

\newcommand{\tend}{\longrightarrow}

\newcommand\asto{\overset{\mathrm{a.s.}}{\tend}}
\newcommand\eqd{\overset{\mathrm{d}}{=}}

\newcommand\bbR{\mathbb R}
\newcommand\bbC{\mathbb C}
\newcommand\bbN{\mathbb N}
\newcommand\bbT{\mathbb T}

\newcounter{CC} 
\newcounter{cc}
\newcommand\E{\operatorname{\mathbb E{}}}
\renewcommand\P{\operatorname{\mathbb P{}}}

\newcommand\supp{\operatorname{supp}}

\newcommand\ga{\alpha}

\newcommand\gd{\delta}

\newcommand\gf{\varphi}
\newcommand\gam{\gamma}
\newcommand\gG{\Gamma}

\newcommand\gl{\lambda}

\newcommand\go{\omega}

\newcommand\gs{\sigma}
\newcommand\gss{\sigma^2}
\newcommand\gth{\theta}
\newcommand\eps{\varepsilon}

\newcommand\cA{\mathcal A}
\newcommand\cB{\mathcal B}

\newcommand\cF{\mathcal F}
\newcommand\cI{\mathcal I}

\newcommand\cP{\mathcal P}
\newcommand\cS{{\mathcal S}}
\newcommand\cT{{\mathcal T}}
\newcommand\cU{{\mathcal U}}

\newcommand\ett[1]{\boldsymbol1[#1]} 

\def\[#1]{[\![#1]\!]}

\newcommand\qw{^{-1}}

\renewcommand{\=}{:=}

\newcommand\oi{[0,1]}
\newcommand\setoi{\set{0,1}}
\newcommand\U{\textsf{U}}
\newcommand\uoi{\U(0,1)}

\newcommand\dd{\,\textup{d}}

\newcommand{\Lovasz}{Lov\'asz}

\newcommand{\tind}{t_{\mathrm{ind}}}

\newcommand{\cuq}{\overline{\cU}}

\newcommand{\cuoo}{\cU_\infty}

\newcommand{\ci}{\cI}
\newcommand{\ciq}{\overline{\ci}}
\newcommand{\cioo}{\cI_\infty}
\newcommand{\dcut}{\gd_\square}

\newcommand{\exch}{exchangeable}

\newcommand{\gwx}[1]{G(#1,W)}

\newcommand{\gwn}{\gwx{n}}
\newcommand{\gn}{G_{n}}
\newcommand{\gmux}[1]{G(#1,\mu)}

\newcommand{\gmun}{\gmux{n}}
\newcommand{\ggw}{\gG_W}
\newcommand{\ggmu}{\gG_\mu}

\newcommand{\ig}{interval graph}

\newcommand{\ps}{\cP(\cS)}
\newcommand{\psl}{\cP_L(\cS)}
\newcommand{\psr}{\cP_R(\cS)}
\newcommand{\psm}{\cP_m(\cS)}
\newcommand{\mul}{\mu_L}
\newcommand{\mur}{\mu_R}

\newcommand{\nj}{_{ni}}
\newcommand{\ch}{\check}
\newcommand{\sss}{\cS}
\newcommand{\sssq}{{\cS^2}}
\newcommand\sfmu{\ensuremath{(\cS,\cF,\mu)}}
\newcommand{\tWx}[1]{W_{#1}^{\gam_{#1}}}
\newcommand{\bd}{\bar d}
\newcommand{\aig}{$\cA$-intersection graph}
\newcommand{\aigs}{$\cA$-intersection graphs}
\newcommand{\aigl}{$\cA$-intersection graph limit}
\newcommand{\sca}{\sss_{\mathsf{CA}}}
\newcommand{\scao}{\sss_{\mathsf{CA}}^0}
\newcommand{\scg}{\sss_{\mathsf{CG}}}
\newcommand{\scgo}{\sss_{\mathsf{CG}}^0}

\newcommand\oivalued{$0/1$-valued} 
\newcommand\REM[1]{{\raggedright\texttt{[#1]}\par\marginal{XXX}}}

\newcommand\urladdrx[1]{{\urladdr{\def~{{\tiny$\sim$}}#1}}}
\usepackage{graphicx}
\usepackage{subfig}
\graphicspath{ {./figs/}}    %%%%???

\begin{document}
\title{Interval graph limits}
\date{February 11, 2011}
% (typeset \today{} \klockan)} 

\author{Persi Diaconis}
\address{Department of Mathematics,
Stanford University, 
CA 94305, USA}
\email{diaconis@math.stanford.edu}
\urladdrx{http://www-stat.stanford.edu/~CGATES/persi}

\author{Susan Holmes}
\address{Department of Statistics,
Stanford, CA 94305, USA}
\email{susan@stat.stanford.edu}
\urladdrx{http://www-stat.stanford.edu/~susan/}
%\thanks{SH was supported by grants NSF DMS-02-41246
%and NIGMS R01GM086884-2.}

\author{Svante Janson}
\address{Department of Mathematics, Uppsala University, PO Box 480,
SE-751~06 Uppsala, Sweden}
\email{svante.janson@math.uu.se}
\urladdrx{http://www.math.uu.se/~svante/}

%\keywords{<keywords>}
\subjclass[2000]{05C99} 
%{60C05 (68P10,68W40)} %%{Primary: <subject>; Secondary: <subject>}

\begin{abstract} 
We work out the graph limit theory for dense interval graphs.
The theory developed departs from the usual description of a graph limit
as a symmetric function $W(x,y)$
 on the unit square, with $x$ and $y$ uniform on the interval $(0,1)$. 
 Instead, we fix a $W$ and change the underlying distribution of the
 coordinates $x$ and $y$. We find choices such that our limits are
 continuous. 
Connections to random interval graphs are given, including some examples.
We also show a continuity result for the chromatic number and clique number of
interval graphs.
%Analogues for circular-arc graphs and unit interval graphs are briefly
%discussed. 
Some results on uniqueness of the limit description are
 given for general graph limits. 
\end{abstract}

\maketitle

\section{Introduction}\label{S:intro}

%All graphs will be simple and unless otherwise stated finite.
A graph $G$ is an \emph{\ig} if
there exists a collection of intervals $\set{I_i}_{i\in V(G)}$ 
%(in $\bbR$)
such that
there is an edge $ij\in E(G)$ if and only if $I_i\cap I_j\neq0$, for all
pairs $(i,j)\in V(G)^2$ with $i\neq j$.
\begin{example}
Figure 1({\sc a}) %\ref{fig:youdeni}
shows published confidence intervals for the astronomical unit
(roughly the length of the semi-major axis of the earth's elliptical orbit about the sun).
Figure 1({\sc b}) %\ref{fig:youdeng}
shows the corresponding interval graph (data from \citet{Youden}).\\
%\vskip-2cm
\begin{figure}
  \centering
  \vskip-1.2cm
  \subfloat[Confidence Intervals]{\label{fig:youdeni}\includegraphics[width=0.5\textwidth,height= 6.6cm]{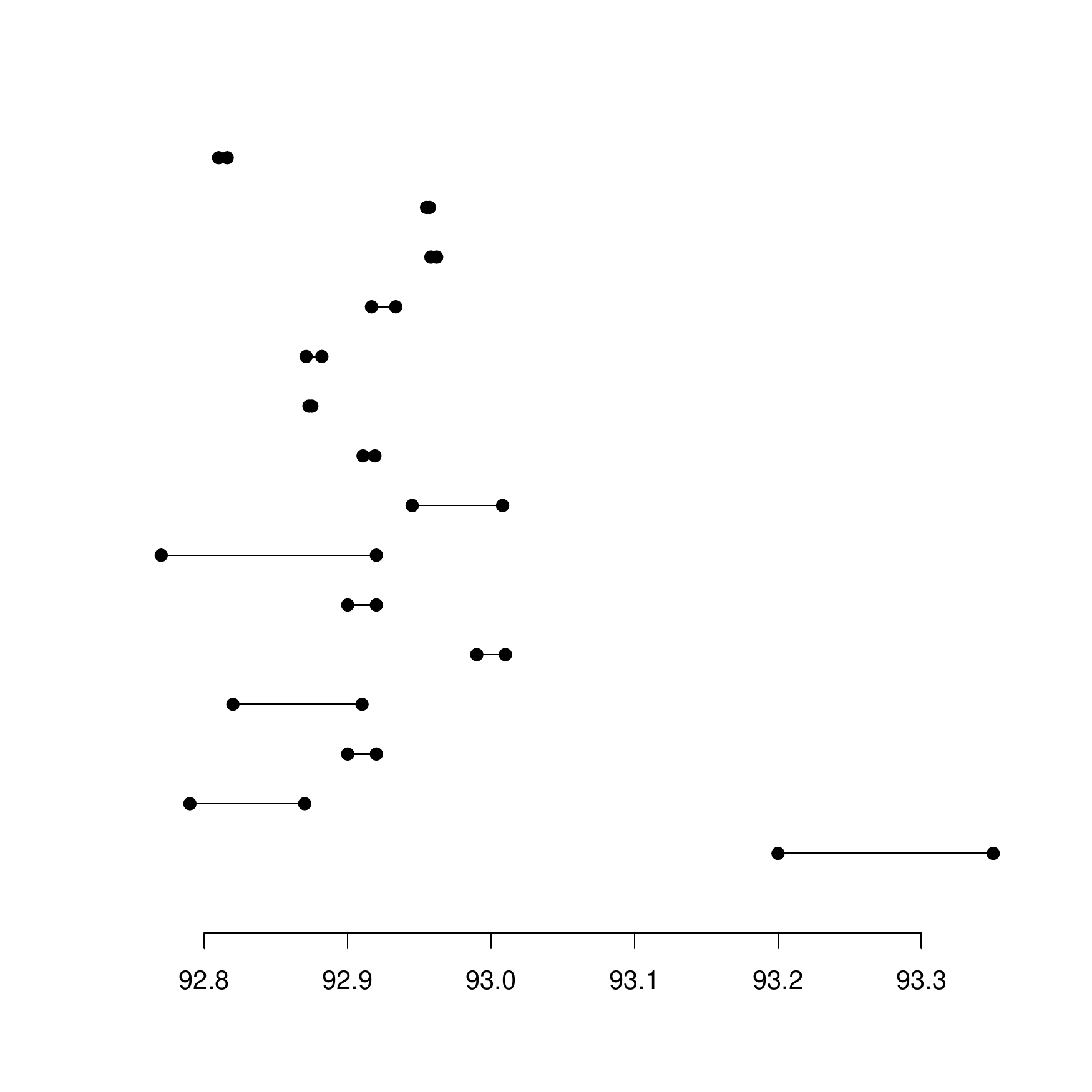}}                
  \subfloat[Interval Graph]{\label{fig:youdeng}\includegraphics[width=0.5\textwidth,height= 6.6cm]{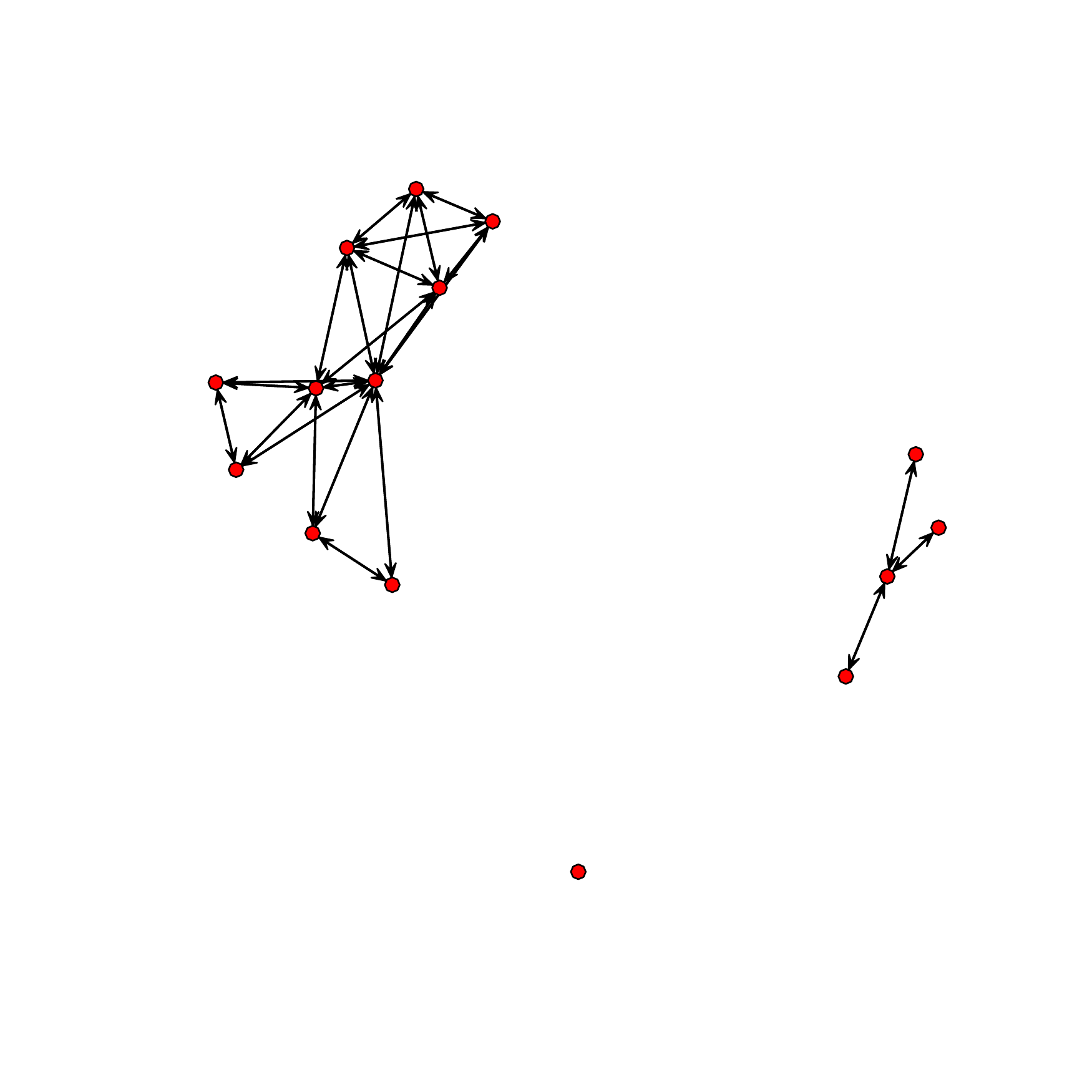}}
  \caption{Building the interval graph for the Youden astronomical constant
	confidence intervals \cite{Youden}. }
  \label{fig:youden}
\end{figure}
It is surprising how many missing edges there are in this graph as these correspond 
to disjoint confidence intervals for this basic unit of astronomy.
Even in the large component, the biggest clique
only has size $4$.
\end{example}
The literature on interval graphs and further examples are given in
\refSS{SS2.1}--\ref{SS2.3} below. 
\refSS{SS2.4} reviews the emerging literature on graph
limits. Roughly, a sequence of graphs $G_n$ is said to converge if the
proportion of edges, triangles and other small subgraphs tends to a
limit. The limiting object is not usually a graph but is represented 
as a symmetric function $W(s,t)$ and a probability measure $\mu$ on a space 
$\cS$.
Again roughly $W(s,t)$ is the chance 
that the limiting graph has an edge from $s$ to $t$, more details will be
provided in \refSS{SS2.4}.  

The main results in this paper combine these two sets of ideas and 
work out the graph limit theory for
interval graphs. 
The intervals in the definition above may be arbitrary intervals of
real numbers $[ a,b ]$, that without loss can be considered inside
$\oi$.
Thus an interval can be identified with a point in 
the triangle $\cS\=\set{[a,b]:0\le a\le b\le 1}$,
see Figure \ref{regionI-II}. 
An interval graph $G_n$ is defined by a set of intervals 
$\set{[a_i,b_i]}_{i=1}^n$, which
may be identified with the empirical
measure $\mu_n=\frac{1}{n}\sum \delta_{(a_i,b_i)}$. 
In \refS{Smain} we show that a sequence of graphs $G_n$
converges 
if the empirical measures $\mu_n$ converge to a limiting probability
$\mu$ in the usual weak star topology, provided $\mu$ satisfies a technical 
condition which we show may be assumed.
The limit of the graphs is specified by a function $W$
defined by
$$
W(a,b;a',b')\=
\begin{cases}
1 & \mbox{ if }[a_,b] \cap [a',b'] \neq \emptyset \\
0 & \mbox{otherwise}\\
\end{cases}
\qquad \mbox{and the limiting } \mu.$$ 
We thus fix $W$ and and simply vary $\mu$ with $\mu$ 
specifying the graph limit; this gives all interval graph limits, but note
that several $\mu$ may give the same graph limit. 
With a na\"ive choice of $\mu$, the assignment of $\mu$ to a graph limit is
not usually 
continuous (as a map from probabilities on $\sss$ to graph limits).
We show that there are several natural choices
of $\mu$ that lead to the same graph limit and result in continuous
assignments.  

\begin{figure}[ht]
\vskip-4cm
\begin{center}
\includegraphics[width=5in]{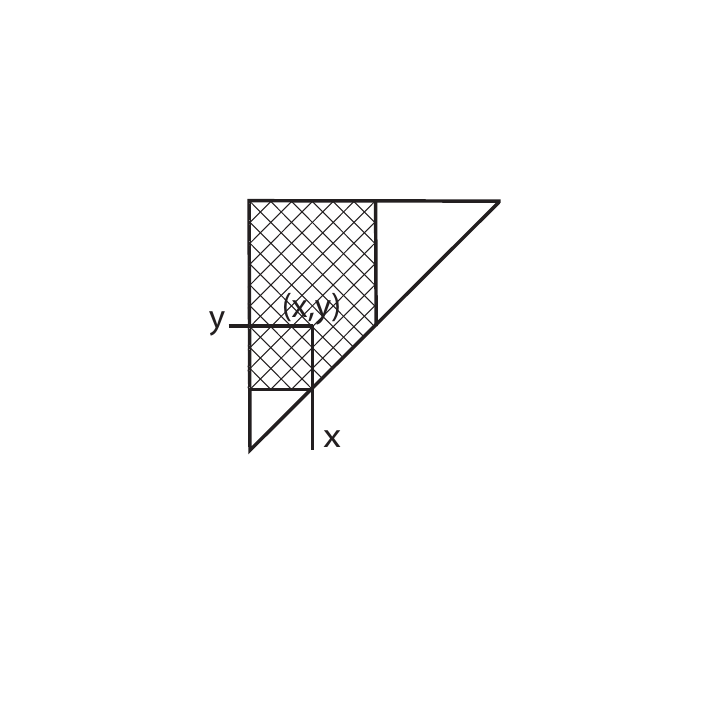}
\vskip-4cm
\caption{For a given $(x,y)$ in $\cS$, the relevant $(x',y')$ that will give an edge
in the intersection graph are in the hatched area.}
\label{regionI-II}
\end{center}
%\vskip-1cm
\end{figure}

 The main theorem is stated in Section \ref{Smain}, and 
results on the chromatic number and clique number are given in \refS{Sclique}.
Some important preliminaries on continuity of the mapping
$\mu\mapsto\ggmu$ are dealt with in \refS{Scont}, and 
\refS{Spf} gives the proof of the main theorem. 
Section
 \ref{Sexamples} discusses some examples of interval graph limits and the
 corresponding random interval graphs.
The parametrization of graph limits is highly non unique; this is seen in
some of the examples in \refS{Sexamples}.
Section
\ref{Sunique} gives a portemanteau theorem which clarifies the connections
between various uniqueness results. This is developed for the general case,
not just interval graphs. 
The problem
of finding a unique ``canonical'' representing measure is still open
in general.
\refS{Spfgo} gives the proofs of the results on clique numbers.
Finally, \refS{Sother} discusses extensions to other classes of intersection
graphs, in particular circular-arc graphs, circle graphs, permutation graphs 
and unit interval graphs.

\section{Background}
This section gives background and references, treating
  interval graphs in Sections  \ref{SS2.1} and \ref{SS2.2}, 
random interval graphs in \refSS{SS2.3}
and   graph limits in Section \ref{SS2.4}--\ref{SSdegrees}.

\subsection{Interval Graphs}\label{SS2.1}
Interval graphs and the closely associated subject of interval orders are a
standard topic in combinatorics. A book length treatment of the subject is
given by \citet{Fishburn}. 
Among many other results, we mention that interval graphs are perfect 
graphs, \ie, the chromatic number equals the size of the largest clique (for
the graph and all induced subgraphs).

Interval graphs are a special case of intersection graphs; 
more generally, we may consider a collection $\cA$ of subsets of some
universe and the class of graphs that can be defined
by replacing intervals by elements of $\cA$ in the definition above.
(We may call such graphs \aigs.)

\citet{McKee}'s book on Intersection Graphs establishes the relation with
intervals. Further literature on the connections between these various graph
classes is in  \citet{Brand} and \citet{Golumbic}.

\subsection{Applications of Interval Graphs}\label{SS2.2}
%\subsubsection*{Past Use of Interval Graphs in Genomics}
 The original question for which interval graphs saw their first application
 was in the structure of genetic DNA. Waterman and Griggs \cite{Waterman}
 and \citet{Klee}
 cite Benzer's original paper from 1959 \cite{Benzer}.
  This is also developed in the papers by Karp \cite{Karp}
 and \citet{GolumbicKS}.
Interval graphs are 
used for censored and truncated data; the interval indicating for instance observed
lifetime (see \citet{Gentleman}
and the R packages
{\tt MLEcens} and {\tt lcens}). 
They also come in when restricting data - like permutations - to certain
observable intervals, this was the motivation behind the astrophysics paper
\citet{Petrosian} and the followup paper
\citet{DiaconisGrahamHolmes}.
For an application of rectangle intersections, see
\citet{rim2002rectangle} and for
sphere intersections see \citet{ghrist2008barcodes}.

\subsection{Random Interval Graphs}\label{SS2.3}
A natural model of {\em random}
interval graphs has $[a_i,b_i]$ chosen uniformly at random inside \oi. 
\citet{Scheinerman:1988} shows that
\begin{align}
\#\text{edges}&=\frac{n^2}{3} + o_p(n^2),\label{1988e}\\
\P\Bigcpar{\frac{\min_v \deg(v)}{\sqrt{n}} \leq x } 
 &\longrightarrow 1-e^{-x^2/2},
\qquad x>0, \label{1988min}
\end{align}
and, if $v$ is a fixed vertex,
\begin{equation} \label{1988d}
  \P\Bigcpar{\frac{\deg(v)}{n} \leq x } \rightarrow 
\begin{cases}
 1-(1-x)\frac{\pi}{2}, & x \geq \frac{1}{2};\\
  1-(1-x)\left\{
  \frac{\pi}{2} -2\cos^{-1}[\frac{1}{\sqrt{2-2x}}]
  \right\}-\sqrt{1-2x},
    & x < \frac{1}{2}.\\
\end{cases}
\end{equation}
He further shows that most such graphs are connected, indeed Hamiltonian,
the chromatic number 
is $\frac{n}{2}+o_p(n)$
and several other things;
see also \citet{JSW} where it is shown that 
the maximum degree is $n-1$ with probability exactly $2/3$
for any $n>1$.
 The chromatic number equals, as said in \refSS{SS2.1},
the size of the largest clique, and this is equivalent
to the random sock sorting problem studied by 
Steinsaltz \cite{Steinsaltz}  and \citet{SJ198}
where more refined results are shown, including asymptotic normality which
for the random interval graph $G_n$ considered here can be written
$(\chi(G_n)-n/2)/\sqrt n\to N(0,1/4)$.

We connect this random interval graph to graph limits in \refE{EU},
where also other models of random interval graphs are considered. 

There has been some followup on this work with \citet{Scheinerman:1990}
introducing
an evolving family of models, and
\citet{Godehardt}
studying independence numbers of random interval graphs for cluster discovery.
\citet{Pippenger} has studied other models with application to allocation in multi-sever queues.
Here, customers  arrive according to a Poisson process, the service time distribution determines an interval length distribution and the intervals,
falling into a given window give an interval graph. For natural models, those graphs
are sparse, in contrast to our present study of dense graphs.

 Finally we give a pointer to an emerging literature on random intersection graphs where subsets of size $d$ 
from a finite set are chosen uniformly for each vertex and there is an edge between two vertices
if the subsets
have a non empty intersection. See results and references in 
\cite{Scheinerman:1999} and 
\cite{Stark}.

%Further, we let from now on
%$W:\cS\times\cS\to\setoi$ be the function
%\begin{equation}\label{w}
%  W(I,J)=\ett{I\cap J\neq\emptyset}.
%\end{equation}
%Then, a graph $G=(V,E)$ is an interval
%graph if and only if there exist intervals $I_v\in\cS$, $v\in V$, such
%that the edge indicators $\ett{vw\in E}=W(I_v,I_w)$, $v\neq w$.
%
%

%(We use the notation $V(G)$ and $E(G)$ for the vertex and edge sets,
%respectively, of a graph $G$; we further let $|G|$ be the number of vertices.)
%See \citet{Fishburn} 
%for various characterizations and properties
%of interval graphs; see also the survey
%\citet{Brand}.
%\marginal{references to some applications?}
%

\subsection{Graph Limits}\label{SS2.4}
This paper studies limits of interval graphs, using the theory of
graph limits introduced by
\citet{LSz} and further developed in
\citet{BCLSVi,BCLSVii} and other papers by various combinations of these
and other authors;
see also \citet{Austin} and \citet{DJ}.
We refer to these papers for the detailed definitions, which may be
summarized as follows
(using the notation of \cite{DJ}).

If $F$ and $G$ are two graphs, then 
$t(F,G)$ denotes the probability that a random mapping $\phi:V(F)\to V(G)$
defines a graph homomorphism, \ie, that $\phi(v)\phi(w)\in E(G)$ when
$vw\in E(F)$. (By a random mapping we mean a mapping uniformly chosen
among all $|G|^{|F|}$ possible ones; the images of the vertices in
$F$ are thus independent and uniformly distributed over $V(G)$, \ie,
they are obtained by random sampling with replacement.)
The basic definition is that a sequence $G_n$ of graphs
converges if $t(F,G_n)$ converges for every graph $F$;
we will use the version in \cite{DJ} where we further assume
$|G_n|\to\infty$.
More precisely, %see \cite{LSz}, \cite{BCLSVi}, \cite{DJ}, 
the (countable and discrete) set $\cU$
of all unlabeled graphs can be embedded in a compact metric space
$\cuq$ such that a sequence 
$G_n\in\cU$ of graphs with $|G_n|\to\infty$
converges in $\cuq$ to some limit $\gG\in\cuq$ if and only if $t(F,G_n)$
converges for every graph $F$. Let $\cuoo\=\cuq\setminus \cU$ be the
set of proper graph limits.
The functionals $t(F,\cdot)$
extend to continuous functions on $\cuq$, 
and an element $\gG\in\cuoo$ is determined by the numbers $t(F,\gG)$.
Hence,
$G_n\to\gG\in\cuoo$ if and
only if $|G_n|\to\infty$ and
$t(F,G_n)\to t(F,\gG)$ for every graph $F$.
(See \cite{BCLSVi,BCLSVii} for other, equivalent, characterizations of
$G_n\to\gG$.) 

We say that a graph limit $\gG\in\cuoo$ is an \emph{interval graph
  limit} if $G_n\to\gG$ for some sequence of interval graphs.
The purpose of the present paper is to study this class of graph
  limits.

  \begin{remark}\label{Rthreshold}
In  \citet{DHJ}, the corresponding problem for the class
  $\cT$ of
  threshold graphs is studied. 
%Note that the threshold graphs form a   subset of the interval graphs, so
%every threshold graph limit is an 
%  interval graph limit. 
%To see this, 
Recall that a graph is a threshold graph 
  \cite{MP}
  if there are real valued vertex labels $v_i$ 
and a threshold $t$ such that $(i,j)$ is an edge
if and only if
$v_i+v_j\leq t$. Equivalently, the graph can be built up sequentially by adding vertices which are either dominating (connected to all previous vertices) or isolated (disjoint from all previous vertices).
Threshold graphs are a subclass of interval graphs; this can be seen from
the sequential description by choosing a sequence of intervals overlapping
all previous or disjoint from all previous intervals as required.  
Thus every threshold graph limit is an interval graph limit. 
The description of threshold
  graph limits in \cite{DHJ} uses special properties of threshold
  graphs, and is of a somewhat different type than the descriptions of
  interval graph limits in the present paper.	
Thus, a threshold graph limit may be represented both as in \cite{DHJ} and
as in the present paper, and the representations will not be the same. (This is
nothing strange, since the representations typically are non-unique.) 
  \end{remark}

Let $\ci\subset\cU$ be the set of all interval graphs, and let
$\cioo\subset\cuoo$ be the set of all interval graph limits;
further, let $\ciq\subset\cuq$ be the closure of $\ci$ in $\cuq$. Then
$\cioo=\ciq\cap\cuoo=\ciq\setminus\ci$.
%thus $\cioo$ is the set of all interval graph limits.
%\ie, $\gG\in\cioo$ if and only if there exists a sequence $G_n$ of interval
%graphs such that $|G_n|\to\infty$ and $G_n\to\gG$. 
Clearly, $\cioo$ is a closed subset of $\cuoo$ and thus a compact
metric space.

A graph limit $\gG\in\cuoo$ may be represented as follows \cite{LSz},
see also \cite{DJ, Austin} for connections to the Aldous--Hoover
representation theory for exchangeable arrays \cite{Kallenberg:exch}.
Let $(\cS,\mu)$ be an arbitrary probability space and let
$W:\cS\times\cS\to\oi$ be a symmetric measurable function.
($W$ is sometimes called \emph{graphon}  \cite{BCLSVi,BCLSVii}, we will use
the alternative \emph{kernel} denomination \cite{SJ253}.) 
Let $X_1,X_2,\dots,$ be an \iid{} sequence of random elements of $\cS$
with common distribution $\mu$. Then there is a (unique) graph limit
$\gG\in\cuoo$ with, for every graph $F$, % and with $k=|F|$,
\begin{equation}\label{t}
  \begin{split}
  t(F,\gG)
&=
\E \prod_{ij\in E(F)} W(X_i,X_j) 
\\
&=
\int_{\cS^{|F|}} \prod_{ij\in E(F)} W(x_i,x_j) 
 \dd\mu(x_1)\dotsm \dd\mu(x_{|F|}).	
  \end{split}
\end{equation}
Further, let, for every $n\ge1$,
$G(n,W,\mu)$ 
be the random graph
obtained by
first taking random $X_1,X_2,\dots, X_n$, and then,
conditionally given $X_1,X_2,\allowbreak\dots,\allowbreak X_n$, 
for each pair $(i,j)$ with $i<j$ letting the edge 
$ij$ appear with probability $W(X_i,X_j)$, (conditionally)
independently for
all pairs $(i,j)$ with $i<j$.
Then the random graph $\gn=G(n,W,\mu)$ converges to $\gG$ \as{} as \ntoo.

Conversely, every graph limit $\gG\in\cuoo$ can be represented in this
way by some such $(\cS,\mu)$ and $W$. (The representation is not
unique, see \refS{Sunique}.)

 \begin{remark}\label{Rcopies}
For any random graph $G(n,W,\mu)$ (not just interval graphs)
the number of copies of any fixed subgraph (e.g.\  triangles)
is a U-statistic, perhaps with extra randomization if $W$ takes on
values other than 0 or 1. Thus central limit theorems
with error estimates and correction terms as well as large deviation
results are available.
\end{remark}

It is usually convenient to fix $(\cS,\mu)$ and let $W:\cS^2\to\oi$
vary; the standard choice of $(\cS,\mu)$ is
the unit interval $\oi$ with Lebesgue measure $\gl$.
(Every graph limit can be represented as in \eqref{t} using this space.)
For interval graphs, however, we find it more natural and convenient
to instead fix $\cS$ and $W$ as follows, and let $\mu$ vary.

There is some flexibility in the definition above of interval
graphs. The intervals in the definition may be arbitrary intervals of
real numbers, or more generally intervals in any totally ordered set,
but we may without changing the class of interval graphs restrict the
intervals to be, for example, closed. 
We may also suppose that all intervals are subsets of $\oi$.
It may sometimes be convenient to allow an empty interval $\emptyset$
(for isolated vertices), but we find it more convenient 
(at least notationally) 
to abstain from this and consider non-empty intervals only.
We will, however, allow ``intervals'' $[a,a]=\set a$ of length 0.

Consequently,  from now and throughout the paper (except where stated
otherwise) we let
$\cS\=\set{[a,b]:0\le a\le b\le 1}$ be the set of closed subintervals
of $\oi$ (non-empty, but allowing intervals of length 0).
$\cS$ is naturally identified with a closed triangle in the plane, and
is thus a compact metric space. (It is the compactness that makes this
space better for our purposes than, for example, the space of all
closed intervals in $\bbR$.)
Further, we let from now on
$W:\cS\times\cS\to\setoi$ be the function
\begin{equation}\label{w}
  W(I,J)=\ett{I\cap J\neq\emptyset}.
\end{equation}
Then, a graph $G=(V,E)$ is an interval
graph if and only if there exist intervals $I_v\in\cS$, $v\in V$, such
that the edge indicators $\ett{vw\in E}=W(I_v,I_w)$, $v\neq w$.

Every probability measure $\mu$ on $\cS$  defines a graph limit
$\gG\in\cuoo$ by \eqref{t}; we denote this graph limit by $\ggmu$.
Similarly, we denote the random graph $G(n,W,\mu)$ constructed from
$(\cS,\mu)$ and $W$ by $\gmun$; this is simply the random interval graph
defined by a random \iid{} sequence of intervals $X_1,X_2,\dots,X_n$ with
distribution $\mu$; we further allow $n=\infty$ here, and let
$G(\infty,\mu)$ be the random infinite graph defined in the same way
by $X_1,X_2,\dots$.
(In \cite{DJ}, the standard situation when $\cS$ and $\mu$ are fixed,
we instead use the notations $\ggw$ and $\gwn$; we will also use that
notation when we discuss general functions $W$ again in \refS{Sunique}.)
Hence, by the general results quoted above,
$\gmun\to\ggmu$ \as{} as \ntoo.
In particular, $\ggmu$ is an interval graph limit: $\ggmu\in\cioo$ for
every probability measure $\mu$ on $\cS$.

\begin{remark}\label{Rghost}
 A graph $G\in\cU$ corresponds to a 'ghost' $\gG_G\in\cuoo$ with
 $t(F,\gG_G)=t(F,G)$ for all $F$ \cite{LSz,DJ}. 
If $G$ is an interval graph represented by a sequence $I_1,\dots,I_n$
 of intervals in $\cS$ (with $n=|G|$), then it follows easily from
 \eqref{t} that
$\gG_G=\ggmu$, where $\mu=\frac1n\sum_1^n \gd_{I_i}$
is the distribution of a random interval chosen uniformly from $I_1,\dots,I_n$.
\end{remark}

Our main theorem (\refT{TI})
gives a converse: every interval graph limit can be
represented by a probability $\mu$ on $\cS$; moreover, we may impose a
normalization. (In fact, we have a choice between three different
normalizations.) However, even with one of these normalizations, the
representing measure is not always unique.

\begin{remark}
  For every measure $\mu$ on $\sss$ we thus have a model $G(n,\mu)$
  of random interval graphs. Different measures $\mu$ give the
  same model (i.e., with the same distribution for every $n$)
if and only if they give the same graph limit $\gG_\mu$, see \refS{Sunique}.
We may thus construct a large number of different models 
of random interval graphs in this way. 
We give a few examples in \refS{Sexamples}.
\end{remark}

\subsection{Degree distribution}\label{SSdegrees}

Suppose that $G_n$ is a sequence of graphs with, for convenience, $|G_n|=n$,
such that $G_n\to\gG$ for a graph limit $\gG$ which is represented by a
kernel $W$ on a probability space $(\sss,\mu)$. 
(In this subsection $W$ and $\sss$ may be arbitrary.)
Let $\bd(G_n)=2e(G_n)/n$
be the average degree of $G_n$. It follows immediately 
$\bd(G_n)/n$ converges to the average 
$\int_\sssq W(x_1,x_2)\dd\mu(x_1)\dd\mu(x_2)$;
in fact, $\bd(G_n)/n=2e(G_n)/n^2=t(K_2,G_n)\to t(K_2,\gG)=\int_\sssq W$.
(Equivalently, the edge density $e(G_n)/\binom n2 \to \int_\sssq W$.)

Moreover, let $\nu(G_n)$ be the normalized degree distribution of $G_n$,
defined as the distribution of the random variable $d_i/n$, where $i$ is a
uniformly random vertex in $G_n$ and $d_i$ its degree. 
Then $\nu(G_n)$ converges weakly (as a probability measure on $\oi$)
to the distribution of the random variable 
$W_1(X)\=\int_\sss W(X,z)\dd\mu(z)$,
where $X$ is a random element of $\sss$ with distribution $\mu$; 
note that $W_1(X)\in\oi$ and that its mean is $\int_\sssq W$.
We can thus regard the distribution of this random variable $W_1(X)$
as the degree distribution of the graph limit; we denote it by $\nu(\gG)$ or
(in our case, where $W$ is fixed) $\nu(\mu)$. 
See for example \cite{DHJ}. 

In particular, for any given $\mu$ on our standard $\sss$, this applies
a.s.\ to the random interval graphs 
$G(n,\mu)$, since $G(n,\mu)\to \gG_\mu$ as said above.

\section{Interval Graph Limits, Theorems}\label{Smain}

Let $\ps$ be the set of probability measures on 
$\cS\=\set{[a,b]:0\le a\le b\le 1}$, equipped with
the standard topology of weak convergence, which makes $\ps$ a compact
metric space.
If $\mu\in\ps$, let $\mul$ and $\mur$ be the marginals of $\mu$
(regarding $\cS$ as a subset of $\bbR^2$), \ie, the probability
measures on $\oi$ induced by $\mu$ and the mappings $\cS\to\oi$
given by $[a,b]\mapsto a$ and $[a,b]\mapsto b$, respectively.

We further consider, both 
as normalizations and
for reasons of continuity, see \refC{Ccont} below,
three subsets of $\ps$: (as above, $\gl$ denotes Lebesgue measure,
\ie, the uniform distribution)
\begin{align}
\psl&\=\set{\mu\in\ps:\mul=\gl},
\label{psl}
\\
\psr&\=\set{\mu\in\ps:\mur=\gl},
\label{psr}
\\
\psm&\=\set{\mu\in\ps:\tfrac12(\mul+\mur)=\gl}.
\label{psm}
\end{align}

We have the following result, which is proved in \refS{Spf}.

\begin{theorem}
  \label{TI}
$\cioo=\set{\ggmu:\mu\in\ps}$.
Moreover, every $\gG\in\cioo$ may be represented as $\ggmu$ where we
further may impose any one of the  normalization conditions
in \eqref{psl}--\eqref{psm}. In other words,
\begin{equation*}
  \cioo
%=\set{\ggmu:\mu\in\ps}
=\set{\ggmu:\mu\in\psl}
=\set{\ggmu:\mu\in\psr}
=\set{\ggmu:\mu\in\psm}.
\end{equation*}
Furthermore, the mapping $\mu\to\ggmu$ is a continuous map of 
each of $\psl$, $\psr$ and $\psm$ onto $\cioo$.
\end{theorem}

The mappings $\psl\to\cioo$, $\psr\to\cioo$, $\psm\to\cioo$
are not injective. We return to this question in Sections \ref{Sexamples}
and \ref{Sunique}. 

The proof in \refS{Spf} also shows the following, which gives an
interpretation of the measure $\mu$.
\begin{theorem}\label{TIG}
  Let $G_n$ be an interval graph, for convenience with $n$ vertices,
 defined by intervals 
$I\nj=[a\nj,b\nj]\subseteq\oi$, $i=1,\dots,n$.
Suppose that, as \ntoo, the empirical measure
\begin{equation}\label{mun}
  \mu_n\=\frac1n\sum_{i=1}^n\gd_{I\nj}
\in \cP(\sss)
\end{equation}
converges weakly to a measure $\mu\in\cP(\sss)$, and suppose
further that $\mul$ and $\mur$ have no common atom. Then $G_n$ converges to
the graph limit $\gG_\mu$.
\end{theorem}

Instead of probability measures $\mu\in\ps$, we may equivalently
consider $\cS$-valued random variables, \ie, random intervals
$[L,R]\in\cS$. 
Each such random interval is given by a pair of random variables
$(L,R)$ with $0\le L \le R \le 1$ (\aex), and conversely.
(Of course, we then only care about the (joint) distribution of $(L,R)$.)
Note that the distribution of $[L,R]$ belongs to $\psl$ [$\psr$]
if and only if $L\sim\uoi$ [$R\sim\uoi$].

\begin{example}
A natural model for a collection of confidence intervals
for a basic physical constant (as Youden's data in the introduction
or the speed of light or the gravitational constant)
has intervals of the form
$[\mu_i-c\sigma_i,\mu_i+c\sigma_i]$
with $\mu_i$ and $\sigma_i$ independently chosen,
$\mu_i$ from a normal $(\mu,\sigma^2)$
distribution and $\sigma_i^2$ from a Chi-squared distribution, $c$ is
computed from the 
normal quantile $q$ and the sample size $c=q_{1-\frac{\alpha}{2}}/\sqrt{n}$,
where $\alpha$ is the target  
type I error.
Here the
intervals are not constrained to $\cS$. A natural transformation using the
distribution 
function $F(x)$ of $\mu_i+c\sigma_i$ yields the random intervals
$[F^{-1}(\mu_i-c\sigma_i),F^{-1}(\mu_i+c\sigma_i)]$, which
correspond to points from a distribution on $\cS$ belonging to our $\P_R(\cS)$.

An example, with $\mu_i\sim N(0,4)$ and $\gss_i\sim \frac{4}{19}\chi^2_{19}$ 
is given in 
Figure \ref{fig:normal}.
\begin{figure}
  \centering
  \vskip-1.2cm
  \subfloat[Confidence Intervals (30) for random normal data generated with $\mu=0,\sigma=2$.]{\label{fig:confint}
  \includegraphics[width=0.5\textwidth,height= 6.6cm]{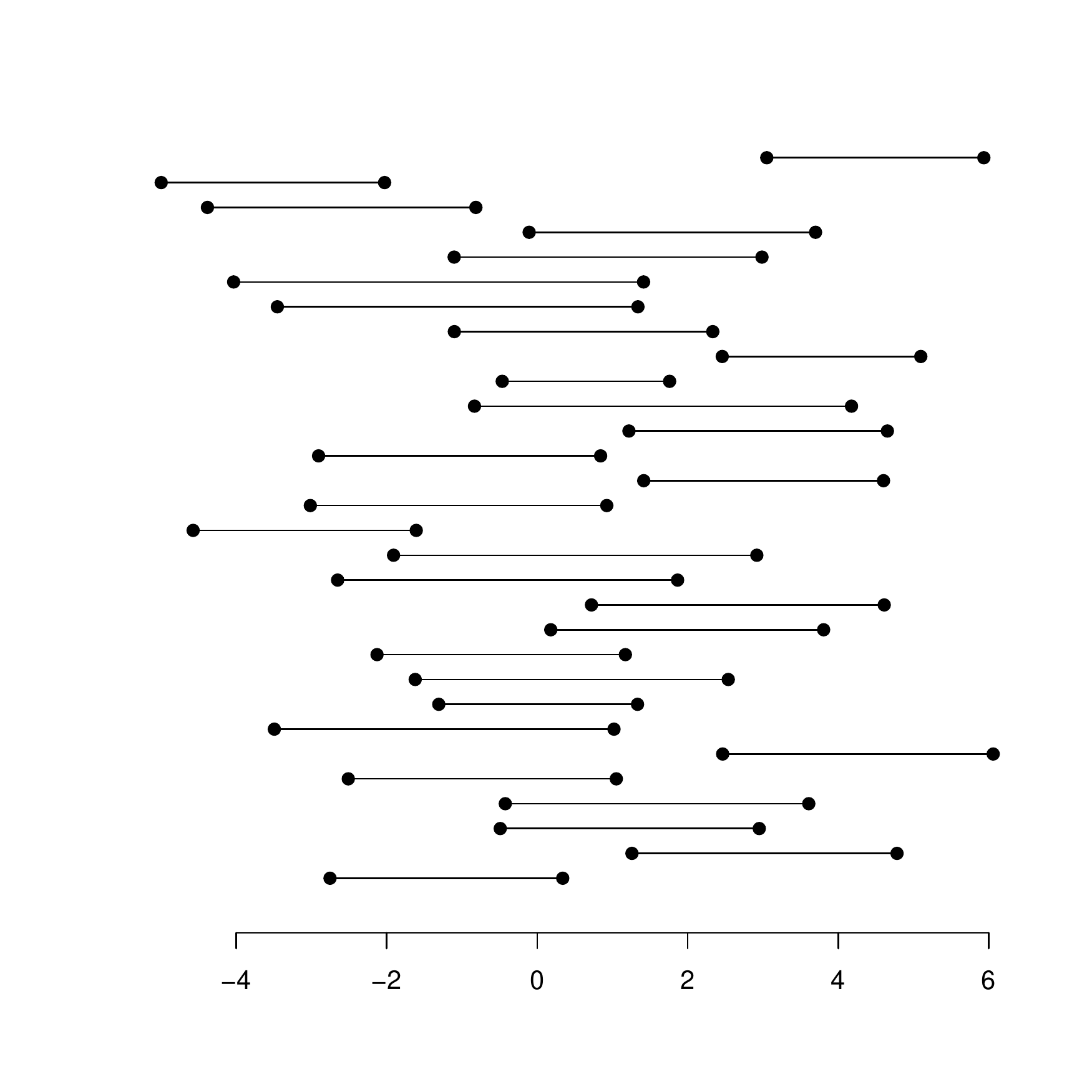}}
%  CIyouden.pdf}}                
  \subfloat[Interval Graph]{\label{fig:normalg}\includegraphics[width=0.5\textwidth,height= 6.6cm]{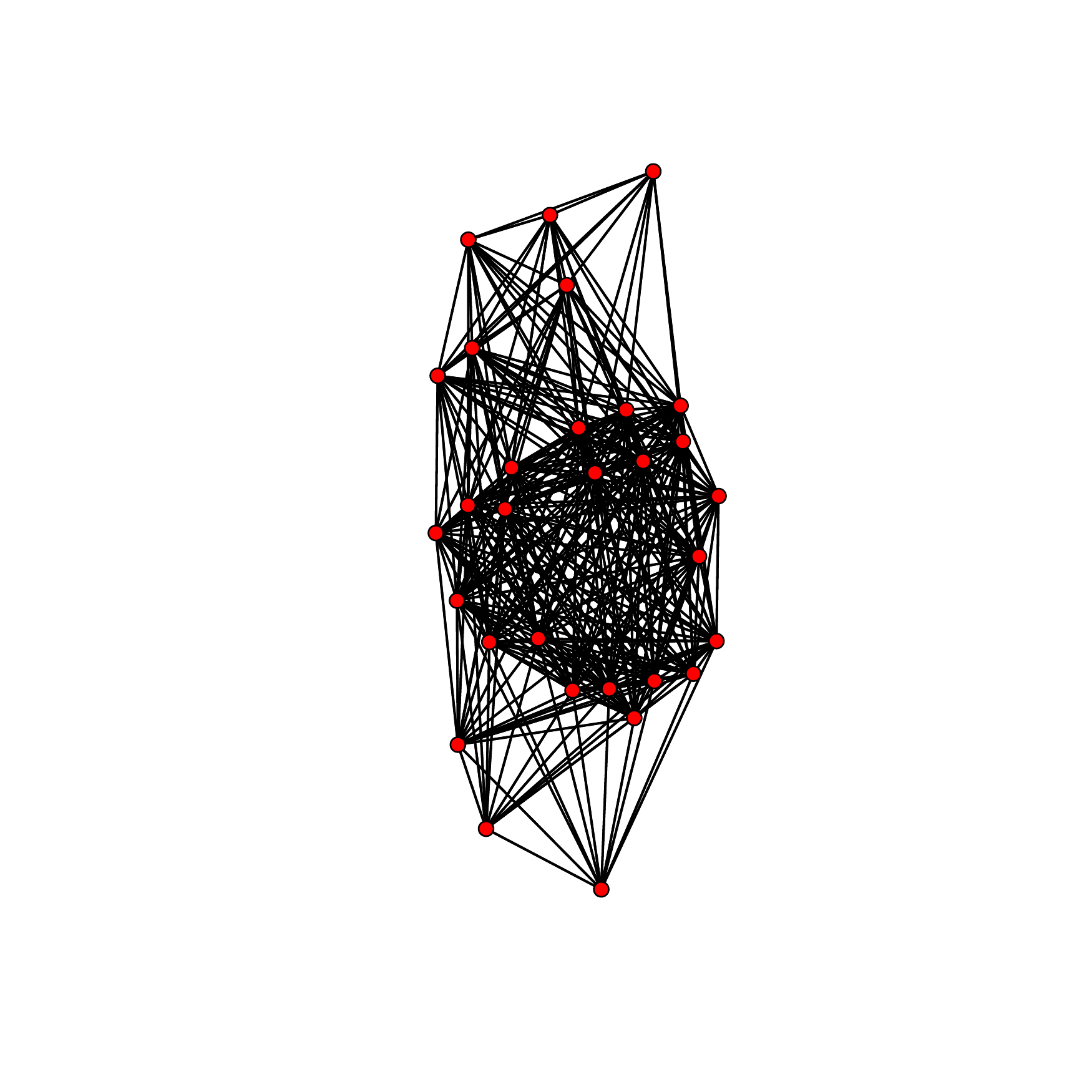}}
  \caption{The interval graph for a sets of Normal intervals.}
  \label{fig:normal}
\end{figure}
Although the graph is not the complete graph as it should be if all $30$ intervals overlapped, the degrees are high and quite even. The degree distribution is:
% > apply(ci10,1,sum)-1
{\small
\begin{verbatim}
 22 20 25 26 23 14 23 23 27 23 26 25 26 27 13 
 23 17 23 20 27 23 14 24 25 25 26 26 17 11 10
\end{verbatim}
}
\end{example}

\begin{remark}\label{Rrefl}
There is an obvious reflection map of $\cS$ onto itself given by
$[a,b]\mapsto[1-b,1-a]$; we denote the
corresponding map of $\ps$ onto itself by $\mu\mapsto\ch \mu$.
(It terms of random intervals $[L,R]$, this is
$[L,R]\mapsto[1-R,1-L]$.)
The reflection map preserves $W$, and it follows that $\gG_{\ch\mu}=\ggmu$.

Note that $\mu\in\psl\iff\ch\mu\in\psr$, and conversely, 
which means that we can transfer results from $\psl$ to $\psr$, and
conversely, by the reflection map; hence it is enough to consider one
of $\psl$ and $\psr$. 
\end{remark}

\begin{remark}\label{Rrandomfree}
As a corollary to \refT{TI}, we see that every limit of interval graphs
  may be represented by a kernel that is \oivalued . (This implies that
  every representing kernel is \oivalued, see \cite{SJ249} for details.)
Graph classes with this property are called \emph{random-free} by
\citet{LSz:regularity}, who among other results gave a graph-theoretic
characterization of such classes.
We have thus shown that the class of interval graphs
is random-free. We will see in Sections \ref{SSCA}--\ref{SSuig} 
that so are the graph classes considered there.
\end{remark}

\section{Cliques and chromatic number}\label{Sclique}

If $G$ is a graph, let $\chi(G)$ be its chromatic number and $\go(G)$ its
clique number, \ie, the maximal size of a clique.
As said in \refSS{SS2.1}, interval graphs are perfect and $\chi(G)=\go(G)$
for them.
If $G$ is an interval graph defined by a collection of intervals
$\set{I_i}$,
it is easily seen that $\go(G)=\max_x\#\set{i:x\in I_i}$. 
We define the corresponding quantity for measures $\mu\in\cP(\sss)$ by
\begin{equation}\label{go}
  \go(\mu)\=
\sup_{a\in\oi}\mu\set{I:a\in I}
=\sup_{a\in\oi}\mu\bigpar{[0,a]\times[a,1]}.
\end{equation}
Thus, 
if $G$ is an interval graph defined by intervals $I_1,\dots,I_n$ in $\sss$,
and $\mu=\frac1n\sum_{i=1}^n\gd_{I_i}$,
then $\go(G)=n\go(\mu)$.

It is easy to see that $a\mapsto \mu\bigpar{[0,a]\times[a,1]}$
is upper semicontinuous; this implies that the supremum in
\eqref{go} is attained.

We will prove the following results in \refS{Spfgo}.

\begin{lemma}
  \label{Lgo=}
If $\mu_1$ and $\mu_2$ are probability measures on $\sss$ 
that are equivalent in the sense that 
$\gG_{\mu_1} =\gG_{\mu_2} $, then $\go(\mu_1)=\go(\mu_2)$.
\end{lemma}

This shows that we can define the clique number $\go(\gG)$ for every
interval graph limit $\gG$ by $\go(\gG_\mu)=\go(\mu)$ for $\mu\in\cP(\sss)$.

\begin{theorem}
  \label{Tchi}
Let $G_n$ be an interval graph, for convenience with $n$ vertices,
and suppose that $G_n\to\gG$ as \ntoo{} for some graph limit $\gG$.
Then
\begin{equation}\label{tchi}
\frac1n\chi(G_n)=\frac1n\go(G_n)\to\go(\gG).
\end{equation}
\end{theorem}

\begin{remark}
Neither $\frac1n\chi$ nor $\frac1n\go$
are continuous functions on the space of all graphs. This may be seen by the following construction: a sequence of dense graphs which tend to the limiting Erd\"os-Renyi graph with $p=1$ (complete graph) but with $\frac1n\chi$ and $\frac1n\go$
converging to limits different from one.
For the construction, let $G_n$ be an Erd\"os-Renyi graph with $p=1-\frac{1}{\sqrt{n}}$.
This converges to the same limit as the sequence of complete graphs $K_n$. However,
an easy argument shows that  $\frac1n\go$  converges to zero. The same example 
can be used to show that $\frac1n\chi$
is not continuous. For this we use the following
\begin{lemma}
For any graph $G$ with $n$ vertices, $\chi(G) \leq (n+\go(G))/2$ 
\end{lemma}
\begin{proof}
Color by picking two non adjacent vertices, giving both the same new color.
Repeat until a connected subgraph of size $m$ remains and give each remaining vertex a separate color. This uses $(n-m)/2+m=(n+m)/2$ colors and $m\leq \go(G)$.
\end{proof}
For the random graphs constructed above, $\go(G)=o(n)$ implies
$\chi(G)\leq \frac{n}{2}+o(n)$. Thus $\frac1n\chi$ is discontinuous.
\end{remark}

\section{Continuity}\label{Scont}

The mapping $\mu\mapsto\ggmu$ of $\ps$ into $\cuoo$ is not
continuous. 
However, the following holds, as we will prove below.
\begin{theorem}
  \label{Tcont}
The mapping $\mu\mapsto\ggmu$ of $\ps$ into $\cuoo$ is continuous at
every $\mu\in\ps$ such that $\mul$ and $\mur$ have no common atom.
Conversely, it is continuous only at these $\mu$.
\end{theorem}

In particular, $\ggmu$ is a continuous function of $\mu$ at every
$\mu$ such that either $\mul$ or $\mur$ is continuous, which yields
the following corollary.

\begin{corollary}
  \label{Ccont}
The mapping $\mu\mapsto\ggmu$ is a continuous map
$\psl\to\cuoo$, $\psr\to\cuoo$ and $\psm\to\cuoo$.
\end{corollary}

To prove \refT{Tcont},
we begin by letting $D_W\subset\cS^2$ be the set of discontinuity
points of $W$.

\begin{lemma}
  \label{LDW}
$D_W=\bigset{([a,b],[c,d]):b=c \text{ \rm or } a=d}$.
\end{lemma}
\begin{proof}
  Obvious.
\end{proof}

\begin{proof}[Proof of \refT{Tcont}]
  Suppose that $\mu_n\to\mu$ in $\ps$ and that $\mul$ and $\mur$ have
  no common atom. Then \refL{LDW} implies that $\mu\times\mu(D_W)=0$,
  and it follows that if $F\in\cU$ and $k=|F|$, then
$\prod_{ij\in E(F)} W(x_i,x_j):\cS^k\to\setoi\subset\bbR$ is
  $\mu^k$-\aex{} continuous. Further, $\mu_n^k\to\mu^k$ in
  $\cP(\cS^k)$, and thus 
  $t(F,\gG_{\mu_n})  \to t(F,\gG_{\mu}) $ by \eqref{t}, see
  \cite[Theorem 5.2]{Bill}. Hence, $\gG_{\mu_n}\to\gG_\mu$ by the
  definition of $\cuoo$.

For the converse (which we will not use), assume that $a$ is a common
atom of $\mul$ and $\mur$. By symmetry we may suppose that $a>0$. 
Let, for $n>1/a$, $a_n\=a-1/n$.
If
$\mu$ has an atom at $[a,a]$, we define $\mu_n$ 
by
moving half of that atom to $[a_n,a_n]$. Otherwise, we 
replace every interval $[c,a]$ with $c\le a_n$ by $[c,a_n]$; this
yields a map $\cS\to\cS$ which maps $\mu$ to a measure $\mu_n$.
It is easy to see, in both cases, that
$\mu_n\to\mu$ but,
using \eqref{t},
\begin{equation*}
 t(K_2,\gG_{\mu_n})=
\int_{\cS^2} W \dd\mu_n\times\dd\mu_n
\not\to
\int_{\cS^2} W \dd\mu\times\dd\mu
=t(K_2,\ggmu). 
\end{equation*}
Hence $\gG_{\mu_n}\not\to\ggmu$.
\end{proof}

\section{Proof of Theorems \ref{TI} and \ref{TIG}}\label{Spf}

\begin{proof}[Proof of \refT{TIG}]
It follows from
\eqref{t}, see \refR{Rghost},
that
\begin{equation}\label{kia}
  t(F,G_n)=t(F,\gG_{\mu_n}), \qquad F\in\cU.
\end{equation}
\refT{Tcont} shows that
$\gG_{\mu_n}\to\ggmu$, \ie, 
$t(F,\gG_{\mu_n})\to t(F,\ggmu)$ for every $F\in\cU$. By \eqref{kia},
this implies $t(F,G_n)\to t(F,\ggmu)$, $F\in\cU$, and thus
$G_n\to\ggmu$.   
\end{proof}

\begin{proof}[Proof of \refT{TI}]
If $\mu\in\ps$, then as said in \refSS{SS2.4}, $\ggmu$ is the
limit \as{} of the sequence $\gmun$ of interval graphs, and thus
$\ggmu\in\cioo$.

Conversely, if $G_n$ is a sequence of interval graphs and
$G_n\to\gG\in\cuoo$, then each $G_n$ is represented by some sequence
of closed intervals $I\nj=[a\nj,b\nj]\subset\bbR$, $i=1,\dots,n$.
By, if necessary, 
increasing the lengths of these interval by small
(and, \eg, random) amounts, we may further assume that for each $n$,
the $2n$ endpoints \set{a\nj,b\nj:1\le i\le n} are distinct.

Using an increasing homeomorphism $\gf_n$ of $\bbR$ onto itself, we may
further assume that the left endpoints \set{a\nj:1\le i\le n} are the
points \set{j/n: 0\le j<n} in some order, and further that all
endpoints $b\nj\le1$. Thus $I\nj\in\cS$ for every $i$. Let 
$\mu_n\in\cP(\sss)$ be 
the corresponding probability measure given by \eqref{mun}.

Since $\cS$ is compact, the sequence $\mu_n$ is automatically tight,
and there exists a probability measure $\mu\in\ps$
such that, at least along a subsequence, $\mu_n\to\mu$. As a
consequence, $\mu_{nL}\to\mul$, and since we have forced $\mu_{nL}$ to
be the uniform measure on the set \set{j/n:j=0,\dots,n-1}, the limit
$\mul=\gl$. Hence $\mu\in\psl$. 

Consequently, \refT{TIG} applies and shows that (along the subsequence)
$G_n\to\ggmu$. Hence $\gG=\ggmu$.

This shows that every $\gG\in\cioo$ equals $\ggmu$ for some
$\mu\in\psl$. The same argument but choosing the homeomorphism $\gf_n$
of
$\bbR$ onto itself such that the right endpoints or all $2n$ endpoints
are evenly spaced in $\oi$ similarly yields $\gG=\ggmu$ with
$\mu\in\psr$ or  $\mu\in\psm$.

This, combined with \refC{Ccont}, completes the proof of \refT{TI}.  
\end{proof}

\section{Examples}\label{Sexamples}

As is well-known, representations as in \refS{S:intro} 
of graph limits by symmetric
measurable functions $W$ on a probability space are far from unique,
see \eg, \cite{LSz,BCLSVi,DJ} and \refS{Sunique}.

In particular, an interval graph limit $\gG\in\cI$ may be represented as
$\ggmu$ for many different $\mu\in\ps$.
For example,
any monotone (increasing or decreasing) homeomorphism $\oi\to\oi$
induces a homeomorphism of $\cS$ onto itself which preserves $W$, and
hence maps any $\mu\in\ps$ to a measure $\mu'$ with $\gG_{\mu}=\gG_{\mu'}$.
(One example of such a homeomorphism of $\cS$ onto itself is the
reflection map in \refR{Rrefl}, induced by the map $x\to1-x$.)
 
If we use one of the normalizations in \eqref{psl}--\eqref{psm} and
consider only $\psl$, $\psr$ or $\psm$, the possibilities are severly
restricted, and we have uniqueness in some cases, but not all.

\begin{example}\label{EKn}
  The complete graph $K_n$ is an interval graph, and can be
  represented by any family of intervals that contain a common point.
The sequence converges to a graph limit $\gG\in\cI$. 
On the standard space $\oi$, $\gG$  is simply represented
by the function $\oi^2\to\oi$ that is identically 1, but we are are
  interested in representations 
as $\ggmu$ for 
$\mu\in\ps$. Clearly, $\gG=\ggmu$ for 
any $\mu\in\ps$ such that
  there exists a point $c\in\oi$ with $\mu$  supported on the
  set \set{[a,b]:a\le c\le b}.

It is easily seen that there is a unique representation with
$\mu\in\psl$; $\mu$ is the distribution of $[U,1]$ with $U\sim \uoi)$.

Similarly (and equivalently by reflection), there is a unique
representation with 
$\mu\in\psr$; $\mu$ is the distribution of $[0,U]$ with $U\sim \uoi$.

However, there are many representations with $\mu\in\psm$; 
these are given by random
intervals $[L,R]$ where $(L,R)$ has any joint distribution with the
marginals $L\sim\U(0,\frac12)$ and $R\sim\U(\frac12,1)$.
\end{example}

\begin{example}\label{E2}
Consider the disjoint union of two complete graphs with $\floor{an}$
  and $n-\floor{an}$ vertices, where $0<a<1/2$.
This sequence of graphs converges as $\ntoo$ to a graph limit that is
  represented by two measures in $\psl$, with corresponding 
  random intervals $[L,R]$ where $L\sim\uoi$ and $R$
is given by either
  \begin{equation*}
R\=
\begin{cases}
  a,& L\le a, \\
  1,& L> a, 
\end{cases}
  \end{equation*}
or the same formula with $a$ replaced by $1-a$.
It can be seen that these two measures are the only measures in $\psl$
representing the graph limit.
(This is an example of a sum of two graph limits; see \cite{SJ213} for
general results on such sums and decompositions.)
\end{example}

\begin{example}\label{E3}
  More generally, let $(p_i)_1^m$ be a finite or infinite sequence of
  positive numbers with sum 1. Let $G_n$ be the interval graph
  consisting of disjoint complete graphs of orders $\floor{np_1}$,
  $\floor{np_2}$, \dots. (Hence, $|G_n|=n-o(n)$.)
It is easily seen that $G_n\to\gG$ for some $\gG\in\cuoo$; thus
  $\gG\in\cI$. (Again, cf.\ \cite{SJ213}.)

To represent $\gG$ as $\ggmu$ with $\mu\in\psl$, let $(J_i)_1^m$ be a
partition of $(0,1]$ into disjoint intervals $J_i=(a_i,b_i]$ with
  $\gl(J_i)=p_i$. Then, if $L\sim\uoi$ and $R$ is defined by $R\=b_i$
  when $L\in J_i$, the random interval $[L,R]$ represents $\gG$. 
If $m<\infty$ and $p_1,\dots,p_m$ are distinct, this gives $m!$
different measures $\mu\in\psl$ representing the same $\ggmu$, since
the intervals $J_i$ may come in any order.
If $m=\infty$, we have an infinite number of different representations.
\end{example}

\begin{example}\label{EU}
  The random interval graph studied by \citet{Scheinerman:1988}, see
  \refSS{SS2.3}, is defined as $G(n,\mu)$ where $\mu\in\cP(\sss)$ is the
  uniform measure on $\sss$; thus $\mu$ has the density $2\dd x\dd y$ on
  $0\le x\le y\le 1$.
Note that the marginal distributions $\mul$ and $\mur$ have densities
$2(1-x)$ and $2x$ on \oi, and are thus not uniform. Hence $\mu\notin\psl$
and $\mu\notin\psr$; however, $\mu\in\psm$.

The integral 
$W_1([x,y])\=\int_\sss W([x,y],J)\dd\mu(J) =1-x^2-(1-y)^2$;
this leads
by \refSS{SSdegrees} and a calculation to the 
degree distribution \eqref{1988d} found by \citet{Scheinerman:1988}.

It is easily seen that $\go(\mu)=1/2$, and thus \refT{Tchi} yields
Scheinerman's result that $\chi(G(n,\mu))/n\to 1/2$ (with
convergence a.s.).

To obtain an equivalent representing measure
$\mu'\in\psr$, we apply the homeomorphism
$x\mapsto x^2$ of $\oi$ onto itself; this measure $\mu'$ has the density 
$(2\sqrt{xy})\qw\dd x\dd y$ on   $\sss=\set{[x,y]:0\le x\le y\le 1}$.

\end{example}

\begin{example} %\label{E1990}
  \citet{Scheinerman:1990} studies another random interval graph model,
  defined by random intervals $[x_i-\rho_i,x_i+\rho_i]$ where $x_i\sim\U(0,1)$
and $\rho_i\sim \U(0,r)$ are independent, and $r>0$ is a parameter.
This is $G(n,\mu)$ where $\mu$ is the uniform distribution on the tilted
rectangle with vertices in $(0,0)$, $(1,1)$, $(1-r,1+r)$, $(-r,r)$; this
rectangle does not lie inside our standard triangle (i.e., the intervals are
not necessarily inside $\oi$), but we may scale it to, for example, 
the rectangle with vertices $(0,\frac{2r}{1+2r})$,
$(\frac{r}{1+2r},\frac{r}{1+2r})$,
$(\frac{r+1}{1+2r},\frac{r+1}{1+2r})$,
$(\frac{1}{1+2r},1)$.
See   \refF{fig:scheinermangraph} for an example. 
\begin{figure}
  \centering
  \hskip-1cm
  \subfloat[Support]{\label{fig:tilted}\includegraphics[width=0.4\textwidth,height= 4.7cm]{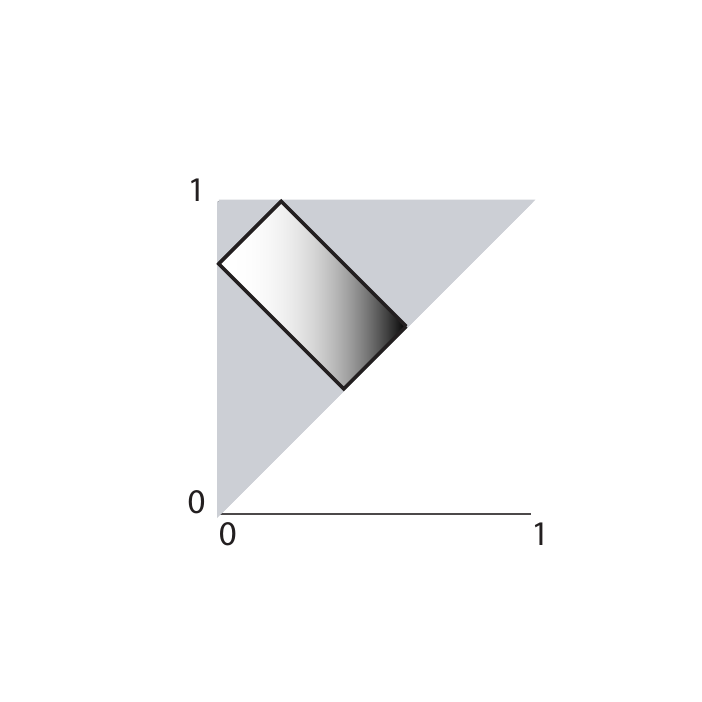}}
  \hskip-1cm                
  \subfloat[Intervals]{\label{fig:schei2}\includegraphics[width=0.4\textwidth,height= 4.7cm]{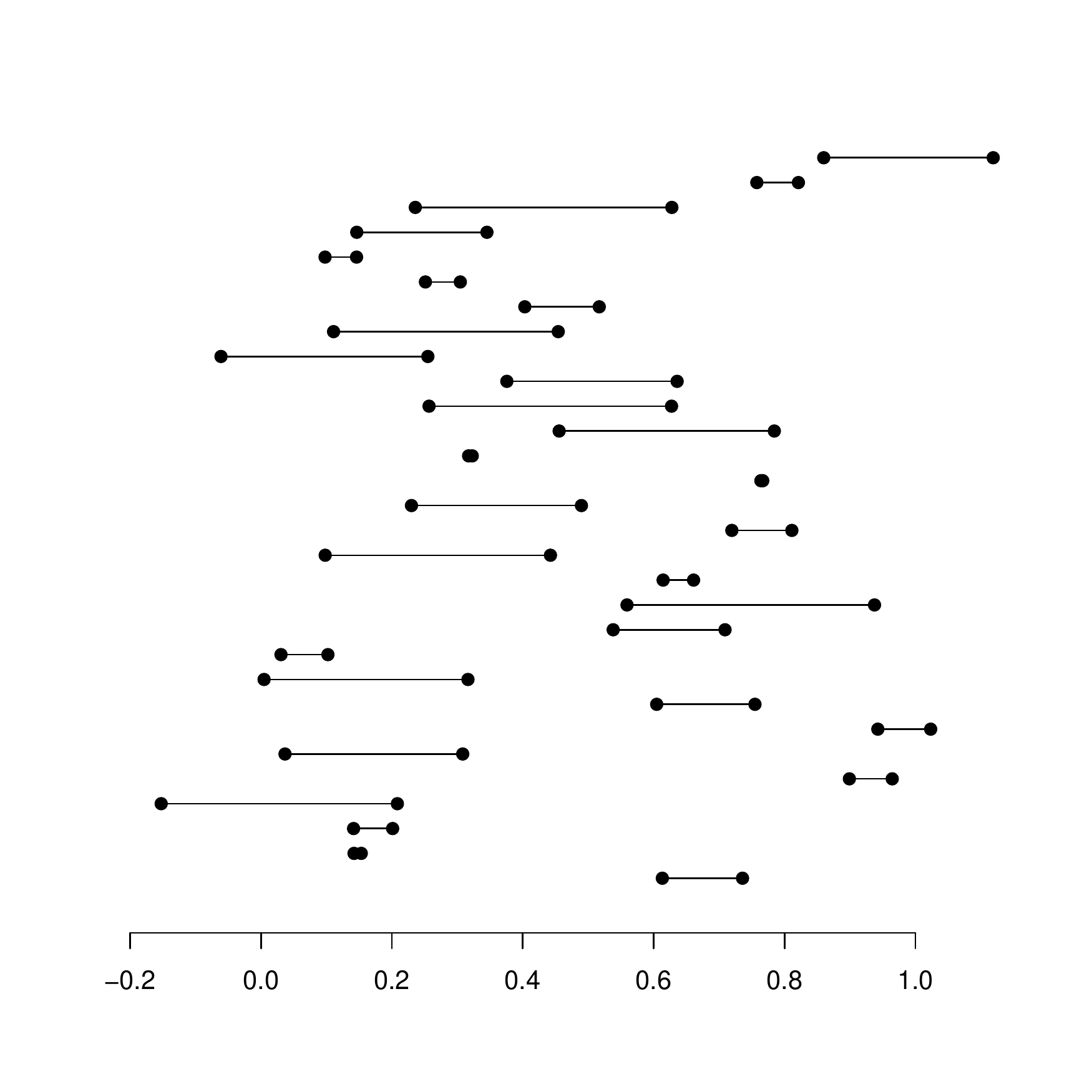}}
  \hskip-1cm                
    \subfloat[Interval Graph]{\label{fig:schei3}\includegraphics[width=0.4\textwidth,height= 4.7cm]{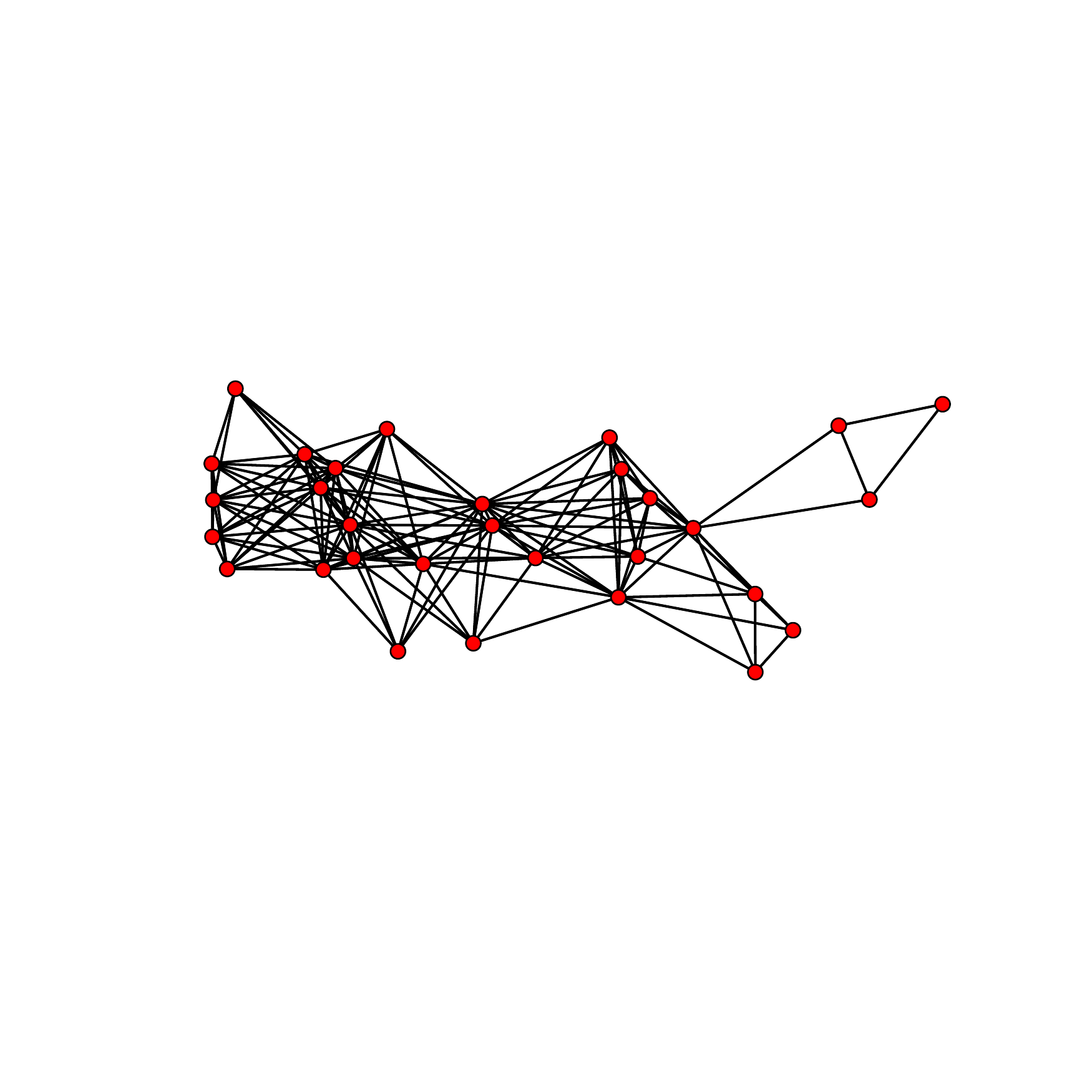}}
  \caption{({\sc A}) shows support of $\mu$: tilted rectangle within $\sss$.
	({\sc B}) shows the intervals with choice of parameters: 30 intervals,
	$r=0.2$. ({\sc C}) shows the corresponding interval graph.} 
  \label{fig:scheinermangraph}
\end{figure}

\end{example}

\begin{example}\label{Er}
Let $0<r\le 1$ and let $\mu$ be uniform on the line 
\set{(x,x+r):0\le x\le 1-r}. This is the set of intervals of length $r$
inside \oi, so 
by scaling we obtain a random set of intervals of length 1 in $\bbR$;
hence the random graph $G(n,\mu)$ is in this case a \emph{unit interval
  graph}, see \refSS{SSuig}.

The degree distribution $\nu(\mu)$, \ie, the asymptotic degree destribution of 
the random graph $G(n,\mu)$, is easily found from \refS{SSdegrees}.
For example, if $r\le \frac13$, then 
%(identifying intervals with their left endpoint in $[0,1-r]$)
$\nu(\mu)$ is the distribution of $W(X)$ with $X\sim\U(0,1-r)$ and
\begin{equation*}
  W_1(x)=
  \begin{cases}
\frac{x+r}{1-r},&0\le x<r,\\
\frac{2r}{1-r},&r\le x\le 1-2r,\\
\frac{1-x}{1-r},&1-2r<x\le 1-r.	
  \end{cases}
\end{equation*}
Thus, $\nu(\mu)$ has a density $2$ on $[\frac{r}{1-r},\frac{2r}{1-r})$ 
and a point mass  $ \frac{1-3r}{1-r}$ at $\frac{2r}{1-r}$.
If $\frac13\le r\le\frac12$, then similarly 
$\nu(\mu)$ has a density $2$ on $[\frac{r}{1-r},{1})$ 
and a point mass  $ \frac{3r-1}{1-r}$ at $1$.
If $r\ge\frac12$, then $G(n,\mu)$ is the complete graph
and $\nu(\mu)$ is a point mass at 1.

The chromatic number is by \refT{Tchi} a.s.\ $\frac{r}{1-r}n+o(n)$ for
$r\le\frac12$ (and trivially $n$ for $r\ge\frac12$).
\end{example}

\begin{example}
\refT{TI} shows that we can build any interval graph limit from a
probability distribution 
$\mu$ on 
$\cS\=\set{[x,y]:0\le x\le y\le 1}$, with the marginal distribution of $\mu$
on the $y$ axis being uniform, \ie, $\mu \in \psr$.
Here is a hierarchy of examples of building such measures
\begin{romenumerate}
\item As in \refE{EKn} for the complete graph $K_n$.
We take $\mu$ to be the uniform distribution on the
$y$ axis. Repeated picks from $\mu$ correspond to intervals
$[0,u_i]$ which all intersect.
\item 
The empty graph $E_n$ is an interval graph corresponding to disjoint
intervals. 
Let $\mu$ be the uniform distribution on the $x=y$ diagonal. 
Repeated picks from $\mu$ yield
intervals $[u_i,u_i]$ which are disjoint with probability 1. 
\item
We may interpolate between these two examples, 
choosing $a$ with $0\leq a \leq 1$ and
$\mu_a$ uniform on the line $\ell_a=\{ (x,y):
x=ay,\, 0 \leq y \leq 1 \}$.
This is done by picking intervals $[aU,U]$ with $U\sim\U(0,1)$ 
so the $y$-margin $U$ is uniform on $[0,1]$.
Now, some pairs of points on the line $\ell_a$ will result in edges 
and some not:\\
 For $[x_1,y_1]$,  $[x_2,y_2]$ in 
 $\cS$,
% the same triangle if
the intervals overlap iff 
 $x_1 \leq x_2\leq y_1$ or
  $x_2 \leq x_1\leq y_2$.
  Equivalently if  
  $x_1\leq x_2$ then $x_2\leq y_1$, or if  
 $x_1\geq x_2$ then $y_2\leq x_1$.
  Here, the points on $\ell_a$ are $[ay_1,y_1]$ and $[ay_2,y_2]$, so there
  is overlap 
 iff %and only if
 $ ay_1 \le ay_2\le y_1$ or $ay_2 \le ay_1\le y_2$, or equivalently,
 $$ ay_1 \leq y_2 \leq y_1/a. 
 $$
 Thus the chance of an edge in this model is 
$\P(aU_1\le U_2\le U_1/a)=2\P(aU_1\le U_2\le U_1)=1-a$.

Moreover, by \refSS{SSdegrees}, the asymptotic degree distribution
$\nu(\mu_a)$ is the distribution of $W_1(U)$, where $U\sim\U(0,1)$ and
\begin{equation*}
W_1(u)=
  \begin{cases}
\bigpar{\frac1a-a}u, & u\ge a,\\
1-au, & a\le u\le 1.	
  \end{cases}
\end{equation*}
This distribution has density $a/(1-a^2)$ on $[0,1-a]$ and $1/(a(1-a^2))$ on
$[1-a,1-a^2]$ (for $a<1$).
The chromatic number is by \refT{Tchi} a.s.\ asymptotic to
$n\go(\mu_a)=(1-a)n$.

 \item
The next example of $\mu\in\psr$ is a  mixture
of uniforms on $\ell_a$, where $a$ has a distribution on $[0,1]$.
The prescription:
\begin{itemize}
\item Pick $a=A$ at random from some distribution on $[0,1]$ and 
\item independently pick uniformly on $\ell_a$,
\end{itemize}
means that we pick intervals $[A,UA]$ with $A$ and $U$ independent
and $U\sim\U(0,1)$, while $A$ has any given distribution.

\item As an extreme example, consider the measure $\mu$ which is a 
$(\theta,  1-\theta)$ 
mixture of uniform on $\ell_0,\ell_1$, with $\gth\in\oi$.
Then, identifying the vertices of
$G(n,\mu)$ with the picked points in $\sss$:
\begin{itemize}
\item None of the points on $\ell_1$ have an edge between them. 
\item All of the points on the line $\ell_0$ have edges between them.
\item Pairs of points, one from $\ell_0$, one from $\ell_1$ have an edge with probability
$1/2$, but not independently.
More precisely, there is an edge between $(0,u_1)$ and $(u_2,u_2)$ iff
$u_1\ge u_2$. 
\end{itemize}
It is easily seen that in this case, the random interval graph $\gmun$ is a
threshold graph, see \refR{Rthreshold}; we may give $(0,u)$ label $u$ and
$(u,u)$ label $-u$ and take the threshold $t=0$. 
(By \cite[Corollary 6.7]{DHJ}, $\gmun$ equals the random graph $T_{n,\theta}$
defined in \cite{DHJ}.)
Hence  $\gG_\mu$ is a
threshold graph limit in this case.
(It is an open problem to characterize all $\mu\in\ps$ such that $\gG_\mu$
is a threshold graph limit.)

\item Uniform intervals: 
As said in \refE{EU}, the uniform distribution on $\sss$ does not belong to
$\psr$, but it is equivalent to the distribution with density
$(2\sqrt{xy})^{-1}\dd x\dd y$ which does. A change of variables to
$(a,y)\in\oi^2$ with $a=x/y$ yields the density $\frac12 a^{-1/2}\dd a\dd y$,
so this is of the type studied here, with $a$ having the $B(\frac12,1)$
distribution with density $\frac12 a^{-1/2}\dd a$.
\end{romenumerate}
\end{example}

\section{Uniqueness}\label{Sunique}

We state a general equivalence theorem for representation of graph
limits (not necessarily interval graph limits) 
by symmetric measurable functions. 
We therefore allow rather general probability spaces 
$(\cS_1,\mu_1)=(\cS_2,\mu_2)$ 
and general
symmetric functions $W_i:\cS_i^2\to\oi$ on them.
In the standard case 
$(\cS_1,\mu_1)=(\cS_2,\mu_2)=(\oi,\gl)$, 
parts \ref{TUgg}--\ref{TUcut} of the theorem  are given in \cite{DJ} as a
consequence of Hoover's equivalence theorem for representations of
exchangeable arrays \citet[Theorem 7.28]{Kallenberg:exch}.
Other similar results are given by \citet{BR} and
\citet{BCL:unique}; in particular, 
\ref{TUtwin1} and \ref{TUtwin2} below are
modelled after similar results in \citet{BCL:unique}.
A similar theorem is stated in \citet{SJ249}, and 
an almost identical theorem in the related case of partial orders is given in
\citet{SJ224}. 

We first introduce more notation.
If $W_2:\cS_2^2\to\oi$ and $\gf:\cS_1\to\cS_2$, then
$W_2^\gf(x,y)\=W_2(\gf(x),\gf(y))$. 

A \emph{Borel space} is a measurable space
$(\sss,\cF)$ that is isomorphic to a Borel subset of
$\oi$, see \eg{} \cite[Appendix A1]{Kallenberg} and \citet{Parthasarathy}. 
In fact,
a Borel space is either isomorphic to $(\oi,\cB)$ or it is
countable infinite or finite. Moreover, every Borel subset of a Polish
topological space (with the Borel $\gs$-field) is a Borel
space.
A \emph{Borel probability space} is a probability space
$\sfmu$ such that $(\sss,\cF)$ is a Borel space.

If $W'$ is a symmetric function $\sssq\to\oi$, where
$\sss$ is a probability space, we say following
\cite{BCL:unique} that
$x_1,x_2\in\sss$ are \emph{twins} (for $W'$) if
$W'(x_1,y)=W'(x_2,y)$ for
\aex{} $y\in\sss$. We say that $W'$ is
\emph{almost twin-free} if there exists a null set
$N\subset\sss$ such that there are no twins
$x_1,x_2\in\sss\setminus N$ with $x_1\neq x_2$.

In the theorem and its
proof, we assume that $\oi$ is equipped with the measure $\gl$, and
$\cS_j$ with $\mu_j$; for simplicity we do not always repeat this.

\begin{theorem}\label{TU}
Suppose that $(\cS_1,\mu_1)$ and $(\cS_2,\mu_2)$ are two Borel
  probability spaces and that $W_1:\cS_1^2\to\oi$ and
  $W_2:\cS_2^2\to\oi$ are two symmetric measurable functions, and let\/
  $\gG_1,\gG_2\in\cuoo$ be the corresponding graph limits.
Then the following are equivalent.
  \begin{romenumerate}
\item\label{TUgg}
$\gG_1=\gG_2$ in\/ $\cuoo$.
\item\label{TUt}
$t(F,\gG_{1})=t(F,\gG_{2})$ for every graph $F$.
\item\label{TUgoo}
The \exch{} random infinite graphs $G(\infty,W_1)$ and $G(\infty,W_2)$
have the same distribution.
\item\label{TUgn}
The random graphs $G(n,W_1)$ and $G(n,W_2)$
have the same distribution for every finite $n$.
\item\label{TUphi}
There exist measure preserving maps $\gf_j:\oi\to\cS_j$, $j=1,2$, such that
$W_1^{\gf_1}=W_2^{\gf_2}$ \aex, \ie,
$W_1\bigpar{\gf_1(x),\gf_1(y)}=W_2\bigpar{\gf_2(x),\gf_2(y)}$ \aex{}
on $\oi^2$.
\item\label{TUpsi}
There exists a measurable mapping $\psi:\cS_1\times\oi\to\cS_2$ 
that maps $\mu_1\times\gl$ to $\mu_2$
such that
$W_1(x,y)=W_2\bigpar{\psi(x,t_1),\psi(y,t_2)}$ for \aex{}
$x,y\in\cS_1$ and  $t_1,t_2\in\oi$.
\item\label{TUcut}
$\dcut(W_1,W_2)=0$, where $\dcut$ is the cut metric defined in
  \cite{BCLSVi} (see also \cite{BR}).
  \end{romenumerate}

If further\/ $W_2$ is almost twin-free, then these are also equivalent to:
\begin{romenumerateq}
\item\label{TUtwin1}
There 
exists a measure preserving map $\gf:\sss_1\to\cS_2$ such
that
$W_1=W_2^{\gf}$ \as, \ie{}
$W_1{(x,y)}=W_2\bigpar{\gf(x),\gf(y)}$ \aex{}
on $\sss_1^2$.  
\end{romenumerateq}

If both\/ $W_1$ and\/ $W_2$ are almost twin-free, then these are also
equivalent to: 
\begin{romenumerateq}
\item\label{TUtwin2}
There exists a measure preserving map $\gf:\sss_1\to\cS_2$ such
that
$\gf$ is a bimeasurable bijection of
$\sss_1\setminus N_1$ onto $\sss_2\setminus
N_2$ for some null sets $N_1\subset\sss_1$ and
$N_2\subset\sss_2$, and 
$W_1=W_2^{\gf}$ \as, \ie{}
$W_1{(x,y)}=W_2\bigpar{\gf(x),\gf(y)}$ \aex{}
on $\sss_1^2$.  
If further $(\sss_2,\mu_2)$ has no atoms, for example if
$\sss_2=\oi$, then we may take $N_1=N_2=\emptyset$.
\end{romenumerateq}
 
\end{theorem}

Note that \ref{TUgg}$\implies$\ref{TUgn} implies that we can uniquely define
the random graphs 
$G(n,\gG)$ for any graph limit $\gG$. 

\begin{proof}
\ref{TUgg}$\iff$\ref{TUt} holds by our definition of graph limits.

Next, the equivalences
\ref{TUgg}$\iff$\ref{TUt}$\iff$\ref{TUgoo}$\iff$\ref{TUgn}$\iff$\ref{TUphi}$\iff$\ref{TUpsi}$\iff$\ref{TUcut}
where shown in \cite{DJ} in the special (but standard) case
$(\cS_1,\mu_1)=(\cS_2,\mu_2)=(\oi,\gl)$.
Since every Borel space is either finite, countably infinite or
(Borel) isomorphic to $\oi$, it is easily seen that there exist
measure preserving maps $\gam_j:\oi\to\cS_j$, $j=1,2$.
Then $W_j^{\gam_j}:\oi^2\to\oi$, and 
%$\tW_j(x,y)\=W_j\bigpar{\gam_j(x),\gam_j(y)}$.
it is easily seen that $\gG_j\=\gG_{W_j}=\gG_{W_j^{\gam_j}}$ and 
$G(n,W_j)\eqd G(n,W_j^{\gam_j})$ for $n\le\infty$, 
and further $\dcut(W_j,W_j^{\gam_j})=0$;
hence
\ref{TUgg}$\iff$\ref{TUt}$\iff$\ref{TUgoo}$\iff$\ref{TUgn}$\iff$\ref{TUcut} 
by the corresponding
results for $\oi$. 

If \ref{TUgg}--\ref{TUgn} 
hold, then by \ref{TUphi}
for $\oi$, there exist measure preserving functions $\gf'_j:\oi\to\oi$
such that 
$\tWx1\bigpar{\gf'_1(x),\gf'_1(y)}=\tWx2\bigpar{\gf'_2(x),\gf'_2(y)}$ \aex{},
and thus \ref{TUphi} holds with $\gf_j\=\gam_j\circ\gf'_j$.

Conversely, if \ref{TUphi} holds, then
$G(n,W_1)\eqd G(n,W_1^{\gf_1})= G(n,W_2^{\gf_2})\eqd G(n,W_2)$ for
every $n\le\infty$; thus \ref{TUphi}$\implies$\ref{TUgoo},\ref{TUgn}.

\ref{TUpsi}$\implies$\ref{TUgoo},\ref{TUgn} is similar.

\ref{TUgoo}$\implies$\ref{TUpsi}: Assume \ref{TUgoo}. Then 
$G(\infty,\tWx1)\eqd G(\infty,\tWx2)$, so by the result for $\oi$,
there exists a measure preserving function $h:\oi^2\to\oi$ such that
$\tWx1(x,y)=\tWx2\bigpar{h(x,z_1),h(y,z_2)}$ for \aex{}
$x,y,z_1,z_2\in\oi$. 
By \cite[Lemma 7.2]{SJ224} (applied to $(\cS_1,\mu_1)$ and $\gam_1$),
there exists a measure preserving map $\ga:\cS_1\times\oi\to\oi$ such that
$\gam_1(\ga(s,u))=s$ a.e.
Hence, for \aex{} $x,y\in\cS_1$ and $u_1,u_2,z_1,z_2\in\oi$,
\begin{equation*}
  \begin{split}
W_1(x,y)
&=
W_1\bigpar{\gam_1\circ \ga(x,u_1),\gam_1\circ \ga(y,u_2)}	
=
\tWx1\bigpar{ \ga(x,u_1), \ga(y,u_2)}	
\\&
=
\tWx2\bigpar{h(\ga(x,u_1),z_1),h(\ga(y,u_2),z_2)}
\\&
=
W_2\bigpar{\gam_2\circ h(\ga(x,u_1),z_1),\gam_2\circ h(\ga(y,u_2),z_2)}.
  \end{split}
\end{equation*}
Finally, let $\beta=(\beta_1,\beta_2)$ be a measure preserving map
$\oi\to\oi^2$, 
and define
$\psi(x,t)\=\gam_2\circ{h\bigpar{\ga(x,\beta_1(t)),\beta_2(t)}}$.

\ref{TUpsi}$\implies$\ref{TUtwin1}:
Since, for \aex{} $x,y,t_1,t_2,t_1'$,
\begin{equation*}
W_2\bigpar{\psi(x,t_1),\psi(y,t_2)}=
W_1(x,y)=W_2\bigpar{\psi(x,t_1'),\psi(y,t_2)}  
\end{equation*}
and $\psi$ is  
measure preserving, it follows that for \aex{}
$x,t_1,t_1'$, $\psi(x,t_1)$ and $\psi(x,t_1')$
are twins for $W_2$. If $W_2$ is almost twin-free, with exceptional null
set $N$, then further
$\psi(x,t_1),\psi(x,t_1')\notin N$ for \aex{}
$x,t_1,t_1'$, since $\psi$ is measure preserving, and consequently
$\psi(x,t_1)=\psi(x,t_1')$ for \aex{} $x,t_1,t_1'$.
It follows that we can choose a fixed $t_1'$ (almost every choice
will do) such that 
$\psi(x,t)=\psi(x,t_1')$ for \aex{} $x,t$. Define
$\gf(x)\=\psi(x,t_1')$.
Then $\psi(x,t)=\gf(x)$ for \aex{} $x,t$, which in particular
implies that $\gf$ is measure preserving, and \ref{TUpsi}
yields $W_1(x,y)=W_2\bigpar{\gf(x),\gf(y)}$ a.e.

\ref{TUtwin1}$\implies$\ref{TUtwin2}:
Let $N'\subset\sss_1$ be a null set such that if
$x\notin N'$, then $W_1(x,y)=W_2(\gf(x),\gf(y))$
for \aex{} $y\in\sss_1$.
If
$x,x'\in\sss_1\setminus N'$ and
$\gf(x)=\gf(x')$, then $x$ and $x'$ are twins
for $W_1$. Consequently, if $W_1$ is almost twin-free
with exceptional null set $N''$, then $\gf$ is injective
on $\sss_1\setminus N_1$ with $N_1\=N'\cup N''$. 
Since $\sss_1\setminus N_1$ and
$\sss_2$ are Borel spaces, the injective map
$\gf:\sss_1\setminus N_1\to\sss_2$ has measurable
range and is a bimeasurable bijection 
$\gf:\sss_1\setminus N_1\to\sss_2\setminus N_2$
for some measurable set $N_2\subset\sss_2$. Since
$\gf$ is measure preserving, $\mu_2(N_2)=0$. 
 
If $\sss_2$ has no atoms, we may take an uncountable null set
$N_2'\subset\sss_2\setminus N_2$. Let
$N_1'\=\gf\qw(N_2')$. Then $N_1\cup N_1'$ and
$N_2\cup N_2'$ are uncountable Borel spaces so there is a
bimeasurable bijection $\eta:N_1\cup N_1'\to N_2\cup
N_2'$. Redefine $\gf$ on $N_1\cup N_1'$ so that
$\gf=\eta$ there; then $\gf$ becomes a bijection $\sss_1\to\sss_2$.

\ref{TUtwin1},\ref{TUtwin2}$\implies$\ref{TUphi}: Trivial. 
\end{proof}

We apply this general theorem to the case $\sss_1=\sss_2=\sss$ and $W_1=W_2=W$.

\begin{corollary}
  \label{Cunique}
Let $\mu_1,\mu_2\in\ps$.
Then, $\gG_{\mu_1}=\gG_{\mu_2}$ if and only if there exists a
measurable map $\psi:\cS\times\oi\to\cS$ that maps
$\mu_1\times\gl\to\mu_2$ such that for $\mu_1$-\aex{} intervals
$I,J\in\cS$ and \aex{} $t,u\in\oi$,
\begin{equation}\label{un}
  I\cap J\neq\emptyset
\iff
\psi(I,t)\cap\psi(J,u)\neq\emptyset.
\end{equation}
\end{corollary}

This result is still not completely satisfactory, and it leads to a number
of open questions:
\begin{problems}\label{PPunique}
\begin{thmenumerate}
\item 
The simple case is when the mapping $\psi$ in \refC{Cunique} does not depend
on the second variable at all; in other words, when there exists a
measurable  map $\gf:\sss\to\sss$ that maps $\mu_1$ to $\mu_2$ 
such that for $\mu_1$-\aex{} intervals
$I,J\in\cS$,
\begin{equation}\label{un0}
  I\cap J\neq\emptyset
\iff
\gf(I)\cap\gf(J)\neq\emptyset.
\end{equation}
When is this possible, and when is the extra randomization in \eqref{un}
really needed?

\item 
To simplify the condition further, when is it possible to choose $\psi$ or
$\gf$ such that 
\eqref{un} or \eqref{un0} hold for \emph{all} $I$ and $J$, and not just
almost all?
Note that in \refE{E2}, the two different representing measures are related
by the map $\gf$ defined by $\gf([x,y])=[x+1-a,y+1-a]$ for $y\le a$
and $\gf([x,y])=[x-a,y-a]$ for $y>x\ge a$, and arbitrarily for $x<a<y$;
this $\gf$ satisfies \eqref{un0} for \aex{} $I$ and $J$, but not for all.

\item 
One way to obtain a map $\gf:\sss\to\sss$ that satisfies \eqref{un} for all
$I$ and $J$ is to take 
$\gf([a,b])=[f(a),f(b)]$ for a (strictly) increasing map $f:\oi\to\oi$,
or $\gf([a,b])=[f(b),f(a)]$ for a (strictly) decreasing map $f:\oi\to\oi$.
Are there any other such maps $\gf$?
Again, note that in Examples \ref{E2} and \ref{E3} there are natural maps $\gf$
that satisfy \eqref{un0} for \aex{} $I$ and $J$, but these are given by
functions $f$ that permute subintervals of $\oi$, and are not monotone.  
It seems that this problem is related to connectedness of the random
interval graphs $G(n,\mu_1)$, and also to the question whether there
are several orientations of the complement of these interval graphs,
\cf{} \cite{Fishburn}.
\end{thmenumerate}
\end{problems}

\begin{problem}\label{Pcanonical}
  Is there some additional condition on $\mu$ that leads to a unique
``canonical''  representing measure $\mu\in\cP(\sss)$ 
for each interval graph limit $\gG$? 
\end{problem}

Note that requiring $\mu\in\psl$ yields uniqueness in \refE{EKn} but not in
\refE{E2}. 

\section{Proof of \refT{Tchi}}\label{Spfgo}

We begin by proving a special case. 
\begin{lemma}
  \label{Lgon}
Let $\mu\in\cP(\sss)$. Then
$
\frac1n\go(G(n,\mu))\asto\go(\mu)$
as \ntoo.
\end{lemma}

\begin{proof}
Recall the construction of $G(n,\mu)$ using i.i.d.\
random intervals $I_1,\dots, I_n$ with distribution $\mu$, and let again 
$\mu_n=\frac1n\sum_1^n \gd_{I_i}$ be the corresponding empirical measure.

Let $\eps>0$.
Choose $a$ such that 
$\go(\mu)= \mu\bigpar{[0,a]\times[a,1]}$.
By the law of large numbers, a.s.\ for all large $n$,
\begin{equation}\label{em>}
  \mu_n\bigpar{[0,a]\times[a,1]}
=\frac1n\#\bigset{i\le n:I_i\in[0,a]\times[a,1]}
>\go(\mu)-\eps.
\end{equation}

In the opposite direction, for every $a\in\oi$, 
$\mu\bigpar{[0,a]\times[a,1]}<\go(\mu)+\eps$, 
and thus, for some $\gd=\gd(a)>0$,
$\mu\bigpar{[0,a+\gd]\times[a-\gd,1]}<\go(\mu)+\eps$.
The open intervals $(a-\gd(a),a+\gd(a))$ cover the compact set $\oi$, so we
can choose a finite subcover $(a_j-\gd_j,a_j+\gd_j)$, $j=1,\dots,m$.
By the law of large numbers, a.s.\ for all large $n$,
$\#\set{i\le n:I_i\in[0,a_j+\gd_j]\times[a_j-\gd_j,1]}<n(\go(\mu)+\eps)$
for each $j=1,\dots,m$, which implies that  
\begin{equation}\label{em<}
  \mu_n\bigpar{[0,a]\times[a,1]}
=\frac1n\#\bigset{i\le n:I_i\in[0,a]\times[a,1]}
<\go(\mu)+\eps
\end{equation}
for every $a\in\oi$. Combining \eqref{em>} and \eqref{em<}, we see that
%a.s.\ for all large $n$ 
$\go(\mu)-\eps < \go(\mu_n)<\go(\mu)+\eps$, and the result follows since
$\frac1n\go(G(n,\mu))=\go(\mu_n)$. 
\end{proof}

\begin{proof}[Proof of \refL{Lgo=}]
  By \refT{TU}\ref{TUgg}$\implies$\ref{TUgn}, the random graphs $G(n,\mu_1)$
  and $G(n,\mu_2)$ have the same distribution and the result follows by
  \refL{Lgon}. 
\end{proof}

A direct analytic proof of \refL{Lgo=}
using e.g.\ \refC{Cunique} seems more difficult than
this  argument using random graphs.

\begin{proof}[Proof of \refT{Tchi}]
As in the proof of \refT{TI}, we may (by considering a subsequence)
assume that $G_n$ is defined by intervals 
$I\nj=[a\nj,b\nj]\subseteq\oi$, $i=1,\dots,n$
such that the corresponding empirical measures $\mu_n$ given by \eqref{mun}
converge to a measure $\mu\in\psm$. By \refT{TIG}, $\gG_\mu=\gG$.

Let $a_n\in\oi$ be such that $\go(\mu_n)=\mu_n\bigpar{[0,a_n]\times[a_n,1]}$. 
By considering a further subsequence we may assume that $a_n\to a$ for some
$a\in\oi$. Since $\mu\in\psm$, $\mul$ and $\mur$ are continuous measures and
thus
$\mu\bigpar{\partial([0,b]\times[b,1])}
=\mu\bigpar{[0,b]\times\set{b}\cup\set{b}\times[b,1]}=0$
for every $b\in\oi$. Together with $\mu_n\to\mu$, this implies
\begin{equation}
\mu\bigpar{[0,b]\times[b,1]}
=\lim_\ntoo \mu_n\bigpar{[0,b]\times[b,1]}
\le \liminf_\ntoo \go(\mu_n).
\end{equation}
Moreover, 
a routine argument shows that
\begin{equation}
  \go(\mu_n)=
\mu_n\bigpar{[0,a_n]\times[a_n,1]}
\to \mu\bigpar{[0,a]\times[a,1]}.
\end{equation}
Consequently, $\go(\mu)=\mu\bigpar{[0,a]\times[a,1]}$ and
$\go(\mu_n)\to\go(\mu)=\go(\gG)$.
The result follows for the subsequence since
$\chi(G_n)=\go(G_n)=n\go(\mu_n)$.
The same argument applies to every subsequence of $G_n$, which thus has a
subsubsequence such that \eqref{tchi} holds; this implies that \eqref{tchi}
holds for the full sequence.
%Let $\eps>0$. For large $n$, $a-\eps<a_n<a+\eps$ and then
%$[0,a-\eps]\times[a+\eps,1]
%\subset [0,a_n)\times(a_n,1] 
%\subset [0,a_n]\times[a_n,1] 
%\subset [0,a+\eps]\times[a-\eps,1]$. 
\end{proof}

\section{Other intersection graphs}\label{Sother}

The methods above can be used also for some other classes of intersection
graphs.
In general, for \aig{s} defined using a collection $\cA$ of
sets, we define $W=W_\cA:\cA\times\cA\to\setoi$ by   
\begin{equation}\label{wa}
 W(A,B)=
 \begin{cases}
1 & \text{if } A\cap B\neq\emptyset,\\
0 & \text{if } A\cap B=\emptyset.
 \end{cases}
\end{equation}
We take $\sss=\cA$ 
(equipped with some suitable $\gs$-field) 
and use this fixed
function $W$, just as for the case of interval graphs above. 
 If $\mu$
is any probability measure on $\sss=\cA$, then the random graphs $G(n,\mu)$
are random \aig{s} (and each $\mu$ gives a model of such random graphs); thus
the graph limit $\gG_\mu$ is an \aigl.
The problem whether the converse holds, 
\ie, whether every \aigl{} can be represented as $\gG_\mu$ for some such $\mu$,
is more subtle; we have proved it for interval graphs above,
and our methods apply also to some other cases, 
see Sections \ref{SSCA}--\ref{SSpg} below;
however, the converse is not true in general, see \refSS{SSuig}.
(For a more trivial counterexample, 
let $\cA$ be the countable family of all finite 
subsets of $\bbN$; then every graph is an \aig, but not every graph limit
can be represented by $\gG_\mu$ for a measure $\mu$ on $\cA$, since 
this would imply that the class of all graphs is random-free, see
\refR{Rrandomfree}, a contradiction.)

We leave the general case as an open problem and remark that our methods
seem to work best when the set $\cA$ has a compact topology; however, even
in that case there are problems because the map $\mu\to\gG_\mu$ is in
general not continuous, as seen in \refT{Tcont}.

\begin{problem}
  Find general conditions on $\cA$ that guarantee that every \aigl{} is
  $\gG_\mu$ for some $\mu\in\cP(\cA)$.
\end{problem}

We study a few cases individually. 
Note that the function $W$ depends
on the graph class by the general formula \eqref{wa}.
%; we omit explicit formulas for the different graph classes below.
%These graph classes turn out to be random-free, see \refR{Rrandomfree}. 
For each class one can ask questions similar to 
Problems \ref{PPunique}--\ref{Pcanonical}, study random graphs $G(n,\gG)$
generated by suitable graph limits, and so on; we leave this to the readers.

\subsection{Circular-arc graphs}\label{SSCA}
\emph{Circular-arc graphs} are the intersection graphs defined by letting
$\cA$ be the collection of arcs on the unit circle $\bbT$,
see \cite{Brand,Golumbic,Klee}.
As for interval graphs, we may assume that the arcs are closed, and we allow
arcs of length 0.
We also allow the whole circle as an arc; this is special since it has no
endpoint. This class obviously contain the interval graphs, and the
containment is strict. (For example, the cycle $C_n$ with $n\ge4$ is a
circular-arc graph but not an interval graph.)

For technical reasons, we first regard the whole circle as having two
coinciding (and otherwise arbitrary) endpoints. The space of arcs may then
be identified with
$\scao\=[0,2\pi]\times\bbT$, with $(\ell,e^{\ii\gth})$ corresponding to the
arc \set{e^{\ii t}:t\in[\gth,\gth+\ell]} of length $\ell$.
The argument in the proof of \refT{TI} shows that every circular-arc graph
limit may be represented as $\gG_\mu$ for some measure $\mu\in\cP(\scao)$,
for example with the marginal distribution of $\gth$ uniform on $\bbT$.

To get rid of the artificial endpoints for the full circle, we identify all
points $(2\pi,e^{\ii\gth})$ in $\scao$ and let $\sca$ be the resulting
quotient space; $\sca$ is homeomorphic to the unit disc
$D\=\set{z\in\bbC:|z|\le1}$ with $re^{\ii\gth}\in D$ corresponding to
$(2\pi(1-r),e^{\ii\gth})\in\scao$ and thus $0\in D$ corresponding to the
full circle. (This gives a unique representation of the closed arcs on $\bbT$.)
The quotient map $\scao\to\sca$ preserves $W$, so by mapping
$\mu$ from $\scao$ to $\sca$, we see that the circular-arc graph limits are
exactly the graph limits $\gG_\mu$ for $\mu\in\cP(\sca)$, in analogy with
\refT{TI} for interval graphs. (The main reason that we do not use $\sca$
directly in the proof is that $W$ is not continuous at pairs $(I,J)$ where
$I=\bbT$ and $J$ has length 0.)

\subsection{Circle graphs}\label{SScircle}

\emph{Circle graphs} are the intersection graphs defined by the collection
of chords of the unit circle $\bbT$ \cite[Chapter 11]{Golumbic}. 
We represent a chord by its two endpoints, and first for convenience
consider the endpoints as an ordered pair of points. 
We thus consider the space $\scgo\=\bbT\times\bbT$ (allowing chords of
length 0).
The argument in the proof of \refT{TI} shows that every circle graph
limit may be represented as $\gG_\mu$ for some measure $\mu\in\cP(\scgo)$,
for example with the average of the two marginal distributions on $\bbT$ being 
uniform (in analogy with $\psm$). 

The space of all chords on $\bbT$ really is the quotient space $\scg$ of
$\scgo$ obtained by identifying $(a,b)$ and $(b,a)$ for any $a,b\in\bbT$.
(The resulting compact space is homeomorphic to a M\"obius strip.)
Again, the quotient mapping preserves $W$, so we can map $\mu\in\cP(\scgo)$
to a measure on $\scg$. Consequently, the circle graph limits are the graph
limits $\gG_\mu$ for $\mu\in\cP(\scg)$.

\subsection{Permutation graphs}\label{SSpg}

A graph is a \emph{permutation graph} if we can label the vertices by
$1,\dots,n$ and there is a permutation $\pi$ of $\set{1,\dots n}$ such that
for $i<j$ there is an edge $ij$ if and only if $\pi(i)>\pi(j)$.
It is easy to see that the
{permutation graphs} are the intersection graphs defined by the collection
of all line segments with one endpoint on each of two parallel lines; we may
take $\cA=\oi\times\oi$ with $(a,b)$ representing the line segment between
$(a,0)$ and $(b,1)$ \cite[Chapter 7]{Golumbic}. 

The argument in the proof of \refT{TI} shows that every permutation graph
limit may be represented as $\gG_\mu$ for some measure $\mu\in\cP(\oi^2)$,
for example with the two marginal distributions on $\oi$ both being 
uniform.

\subsection{Unit interval graphs}\label{SSuig}

\emph{Unit interval graphs} are the intersection graphs defined by the
collection $\cA=\set{[x,x+1]:x\in\bbR}$ of unit intervals in $\bbR$.
(Again, we choose the intervals as closed; the collection of open unit
intervals defines the same class of graphs.)
This class coincides with the class of \emph{proper interval graphs},
defined by collections of intervals $I_1,\dots,I_n$ in $\bbR$, with the
additional requirement that no $I_i$ is a proper subinterval of another.
(Or, equivalently, that $I_i\not\subseteq I_j$ for all $i,j$.)
They are also called \emph{indifference graphs}. 
See \cite{Brand,Golumbic,Roberts}.
This is a subclass of all interval graphs and the containment is strict
since $K_{1,3}$ is an interval graph but not a unit interval graph.

The set $\cA$ above is naturally identified with $\bbR$,
with $W(x,y)=1$ when $|x-y|\le1$; thus every
probability measure on $\bbR$ defines a unit interval graph limit.
However, this mapping is \emph{not} onto. In fact, the empty graph $E_n$ is
a unit interval graph, so the limit as \ntoo{} is a unit interval graph
limit; this graph limit $\gG_0$ is defined by the kernel 0 on any
probability space and has $t(K_2,\gG_0)=0$, but if $\mu\in\cP(\bbR)$, then
the corresponding graph limit $\gG_\mu$ has by \eqref{t}
\begin{equation*}
  t(K_2,\gG_\mu)=\iint_{|x-y|\le1}\dd\mu(x_1)\dd\mu(x_2)>0.
\end{equation*}
Thus $\gG_0\neq\gG_\mu$. (Note that if $\mu_n\in\cP(\bbR)$ is a measure
representing 
$E_n$, then necessarily the sequence $\mu_n$ is not tight, and in fact
converges vaguely to 0, so this problem is connected to the non-compactness
of $\bbR$.)

Another approach to unit interval graph limits is to regard them as special
cases of interval graph limits and use the theory developed above to
characterize them using special measures on the triangle 
$\sss=\set{[a,b]:0\le a\le b\le 1}$.
This yields the following theorem.
\begin{theorem}\label{Tuig}
A graph limit\/ $\gG$ is a unit interval graph limit if and only if 
$\gG=\gG_\mu$ for a measure $\mu\in\cP(\sss)$ that has support on some curve 
$t\mapsto\gam(t)=(\gam_1(t),\gam_2(t))\in\sss$ such that $\gam_1(t)$ and
$\gam_2(t)$ are weakly increasing.
\end{theorem}
\begin{proof}
Suppose that $G_n$ is a sequence of unit interval graphs with $G_n\to\gG$.
In the proof of \refT{TI}, the interval representations are modified by
homeomorphisms, and the results are, of course, not unit interval
representations, but they are proper interval representations, \ie, no
interval is a subinterval of another. Thus, the measures $\mu_n$ have the
property that for each $(a,b)\in\sss$, 
$\mu_n\bigpar{[0,a)\times(b,1]}\cdot\mu_n\bigpar{(a,1]\times[0,b)}=0$.
Since $\mu_n\to\mu$ (for a subsequence), the same holds for $\mu$, which
implies that if $a_1<a_2$ and $b_1>b_2$, then $(a_1,b_1)$ and $(a_2,b_2)$
cannot both belong to $\supp\mu$.
(Choose $a=(a_1+a_2)/2$ and $b=(b_1+b_2)/2$.)

Let $E=\set{a+b:(a,b)\in\supp\mu}$. Then $E$ is a closed subset of $[0,2]$
and for each $t\in E$ there is exactly one $(a,b)\in\supp\mu$ with $a+b=t$;
we define $f(t)=a$ and $g(t)=b$ so $f$ and $g$ are functions $E\to\oi$.
Note that $f(t)+g(t)=t$.
If $t_1<t_2$ and $g(t_1)>g(t_2)$, then $f(t_1)<f(t_2)$; thus
$(f(t_1),g(t_1))$ and $(f(t_2),g(t_2))$ are two points in $\supp\mu$
violating the condition above. Consequently, if $t_1<t_2$ then $g(t_1)\le
g(t_2)$, and similarly $f(t_1)\le f(t_2)$. (Since $f(t)+g(t)=t$ this further
implies $f(t_2)-f(t_1)\le t_2-t_1$ and $g(t_2)-g(t_1)\le t_2-t_1$.)
We may now extend $f$ and $g$ to the complement $\oi\setminus E$, \eg{}
linearly in each component, and define $\gam(t)=(f(t),g(t))$.

For the converse, consider the random graph $G(n,\mu)$. This is an interval
graph represented by intervals $I_1,\dots,I_n\in \sss$ that lie on the curve
$\gam$. This is not necessarily a proper interval representation, since two
of the intervals may lie on the same horizontal or vertical part of $\gam$,  
but it is easily seen that it is always possible to obtain a proper interval
representation of the same graph by moving some of the endpoints a
little. Thus $G(n,\mu)$ is a proper interval graph, and thus a unit interval
graph, whence $\gG_\mu$ is a unit interval graph limit.
\end{proof}
Again, the representation by such a measure $\mu$ is not unique. 

\begin{problem}
Is it
possible to make a canonical choice in some way? Is it possible to use a
fixed curve $\gam$?  
\end{problem}

\begin{remark}
$\gG_\mu$ may happen to be a unit interval graph limit also
if $\mu$ is not of the type in \refT{Tuig}; for example if $\mu$ is any
measure supported on $[0,\frac12]\times[\frac12,1]$ when each $G(n,\mu)$ is
the complete graph $K_n$. To characterize all measures $\mu\in\ps$ such that
$\gG_\mu$ is a unit interval graph limit is a different, and open,  problem.  
\end{remark}

The unit interval graphs can also be characterized as the intervals graphs $G$
that do not contain $K_{1,3}$ as an induced subgraph
\cite{Brand,Golumbic,Roberts}. 
In general, for two graphs $F$ and $G$ with $|F|\le|G|$, let $\tind(F,G)$ be
the probability that the induced subgraph of $G$ obtained by selecting $|F|$
vertices uniformly at random is isomorphic to $F$; this number is closely
connected to $t(F,G)$ defined in \refSS{SS2.4} (which loosely speaking
counts subgraphs of $G$ and not just induced subgraphs), see \cite{BCLSVi},
\cite{LSz} 
or \cite{DJ} for details. 
For any fixed $F$, $\tind(F,\cdot)$ extends to graph limits $\gG$
and we have $\tind(F,G_n)\to\tind(F,\gG)$ if $G_n\to\gG$; moreover,
$\tind(F,\gG)$ is a continuous function of $\gG$.
Using this notation, $G$ is a unit interval graph if and only if $G$ is an
interval graph with $\tind(K_{1,3},G)=0$. 

\begin{theorem}\label{Tuig2}
  Let $\gG$ be a graph limit.
Then the following are equivalent:
\begin{romenumerate}
\item 
$\gG$ is a  unit interval graph limit.
\item 
$\gG$ is an interval graph limit and $\tind(K_{1,3},\gG)=0$.
\item
The random graphs $G(n,\gG)$ are unit interval graphs.
\end{romenumerate}
\end{theorem}
\begin{proof}
  (i)$\implies$(ii) is clear by the comments above.

  (ii)$\implies$(iii). Use \refT{TI} and choose a measure $\mu\in\cP(\sss)$
  representing $\gG$. There is a formula analoguous to \eqref{t} for
  $\tind(F,\gG)$, with $\prod_{ij\in E(F)}W(x_i,x_j)$  replaced by 
$\prod_{ij\in E(F)}W(x_i,x_j)\prod_{ij\notin E(F)}(1-W(x_i,x_j))$,
and it follows easily that for any $n\ge|F|$,
$$
\E\tind\bigpar{F,G(n,\gG)}
=
\E\tind\bigpar{F,G(|F|,\gG)}
=\tind(F,\gG).
$$ 
Hence, (ii) implies that $G(n,\gG)$ a.s.\ is an interval graph $G$
with $\tind(K_{1,3},G)=0$, \ie, a unit interval graph. (The case $n<4$ is
trivial.) 

  (iii)$\implies$(i) follows since $G(n,\gG)\to\gG$ a.s.
\end{proof}

Finally, we mention that a related characterization of unit interval graphs
is that they are the graphs that contain no induced subgraph isomorphic to
$C_k$ for any  $k\ge4$, $K_{1,3}$, $S_3$ or $\overline S_3$, where $S_3$ is
the graph on 6 vertices \set{1,\dots,6} with edge set
$\set{12,13,23,14,25,36}$, and $\overline S_3$ is its complement
\cite{Brand}. %Th 7.1.9, due to Wegner
The same argument as in the proof of \refT{Tuig2} yields
(see \cite[Theorem 3.2]{DHJ} for a more general result):
\begin{theorem}\label{Tuig3}
  A graph limit\/ $\gG$ is a unit interval graph limit if and only if
$\tind(F,\gG)=0$ for every  
$F\in\set{C_k}_{k\ge4}\cup\set{K_{1,3},S_3,\overline S_3}$.
%, for  every $k\ge 4$,
%$\tind(C_k,\gG)=\tind(K_{1,3},\gG)=\tind(S_3,\gG)=\tind(\overline S_3,\gG)=0$.
\nopf
\end{theorem}

\begin{acks}
Part of this research was done during the
2007 Conference on Analysis of Algorithms (AofA'07) in
Juan-les-Pins, France,
% near Antibes and Nice, France on June 17-22, 2007.
and in a bus back to Nice after the conference.
\end{acks}

\newcommand\AAP{\emph{Adv. Appl. Probab.} }
\newcommand\JAP{\emph{J. Appl. Probab.} }
\newcommand\JAMS{\emph{J. \AMS} }
\newcommand\MAMS{\emph{Memoirs \AMS} }
\newcommand\PAMS{\emph{Proc. \AMS} }
\newcommand\TAMS{\emph{Trans. \AMS} }
\newcommand\AnnMS{\emph{Ann. Math. Statist.} }
\newcommand\AnnPr{\emph{Ann. Probab.} }
\newcommand\CPC{\emph{Combin. Probab. Comput.} }
\newcommand\JMAA{\emph{J. Math. Anal. Appl.} }
\newcommand\RSA{\emph{Random Struct. Alg.} }
\newcommand\ZW{\emph{Z. Wahrsch. Verw. Gebiete} }
\newcommand\DMTCS{\jour{Discr. Math. Theor. Comput. Sci.} }

\newcommand\AMS{Amer. Math. Soc.}
\newcommand\Springer{Springer-Verlag}
\newcommand\Wiley{Wiley}

\newcommand\vol{\textbf}
\newcommand\jour{\emph}
\newcommand\book{\emph}
\newcommand\inbook{\emph}
\def\no#1#2,{\unskip#2, no. #1,} %(typeset after year) 
\newcommand\toappear{\unskip, to appear}

\newcommand\webcite[1]{\url{#1}}
%   \penalty1\texttt{\def~{{\tiny$\sim$}}#1}\hfill\hfill}
\newcommand\webcitesvante{\webcite{http://www.math.uu.se/~svante/papers/}}
\newcommand\arxiv[1]{\webcite{http://arxiv.org/#1}}

\def\nobibitem#1\par{}

\end{document}